\documentclass[twoside,11pt]{article}
\usepackage{amsmath,amssymb,amsthm,amsfonts}
\usepackage{paralist}
\usepackage{graphics} 
\usepackage{epsfig} 
\usepackage{graphicx}
\usepackage{caption}
 \usepackage[font=scriptsize,labelfont=sf]{caption}
\usepackage{subcaption}
\usepackage{epstopdf}
\usepackage[colorlinks=true]{hyperref}
\usepackage{enumerate}
\usepackage[numbers,sort&compress]{natbib}
\numberwithin{equation}{section}

\newcommand{\be}{\begin{equation}}
\newcommand{\ee}{\end{equation}}

\newcommand{\ben}{\begin{eqnarray*}}
\newcommand{\enn}{\end{eqnarray*}}

\newtheorem{proposition}{Proposition}[section]
\newtheorem{theorem}{\textbf Theorem}[section]
\newtheorem{lemma}{\textbf Lemma}[section]

 \numberwithin{equation}{section}
\newtheorem{remark}{Remark}[section]
\graphicspath{{Images/}}

\begin{document}

\title{{\textbf{Phase transitions and bump solutions of the Keller-Segel model with volume exclusion} }}
\author{Jose A. Carrillo\thanks{Department of Mathematics, Imperial College London, London SW7 2AZ
        ({carrillo@imperial.sc.uk}).}
\and Xinfu Chen \thanks{Department of Mathematics, University of Pittsburgh, 301 Thackeray Hall, Pittsburgh, PA 15260
        ({xinfu@pitt.edu}).}
\and Qi Wang \thanks{Department of Mathematics, Southwestern University of Finance and Economics, 555 Liutai Ave, Wenjiang, Chengdu, Sichuan 611130, China
        ({qwang@swufe.edu.cn}).}
\and Zhian Wang \thanks{Department of Applied Mathematics, Hong Kong Polytechnic University, Hung Hom, Kowloon, Hong Kong
        ({mawza@polyu.edu.hk}).}
\and Lu Zhang \thanks{Department of Mathematics, Southern Methodist University, 6425 Boaz Lane, Dallas TX 75205, USA
         ({luzhang@smu.edu}).}
        }

\maketitle

\begin{abstract}
We show that the Keller-Segel model in one dimension with Neumann boundary conditions and quadratic cellular diffusion has an intricate phase transition diagram depending on the chemosensitivity strength. Explicit computations allow us to find a myriad of symmetric and asymmetric stationary states whose stability properties are mostly studied via free energy decreasing numerical schemes. The metastability behavior and staircased free energy decay are also illustrated via these numerical simulations.
\end{abstract}

\section{Introduction and main results}

Aggregation-diffusion equations are ubiquitous in the modelling of phenomena in mathematical biology; from collective behavior of animal groups \cite{MogilnerEdelstein,TopazBertozzi2} to cell differential adhesion \cite{APS06,MurTog,CHS17} passing through cancer invasion models \cite{GC,DTGC}, being cell movement by chemotaxis one of their most classical applications in mathematical biology \cite{Pat53,KS70,HP09}. We refer to \cite{CCY} for a recent survey of current research in aggregation-diffusion equations.

Among this large class of equations, there is a particular case that has recently attracted lots of attention corresponding to very localized repulsion and attraction due to chemotaxis. More precisely, assume that we have cells whose nuclei are located at positions $\{ x_i \}$, $i=1,\dots,N$. Let us suppose that cells will interact with other cells either by chemoattractive interaction at a fairly long-range distance through the production of a chemoattractant substance $v(x,t)$ or by strong repulsion, if the interparticle cell distance becomes very small due to volume size exclusion constraints around the nuclei. Let us also assume that the localized repulsive forces exerted by cell type $j$ onto cell type $i$ are radial in the direction of the centers of the nuclei and therefore they follow from a radial potential denoted by $W^N$. The basic agent-based model for this ensemble of cells of mass $m$ moving up the gradient of the chemoattractant $v$ with strongly localized repulsion reads as
\begin{align*}
	\dot x_i &= \chi \sum_{j\neq i}\nabla v(x_i) - \frac{m}N\sum_{j\neq i} \nabla W^N (x_i-x_j),
\end{align*}
with $\chi$ the chemosensitivity dimensionless parameter after standard non--dimensionalization. Here, we made the mean-field assumption in order to keep a finite mass $m$ in the limit of large number of agents $N\to \infty$, that is
$$
u(x,t) \simeq \frac{m}N \sum_{i=1}^N \delta_{x_i(t)}
$$
as $N\to \infty$. We now assume that the scaling of this repulsive potential reflects the volume size restriction modeled by localized repulsion \cite{Ol,BV}.  The potential is scaled in $N$ such that $W^N \simeq \delta_0 $ as $N\to \infty$, then taking the limit $N\to\infty$ leads to the following PDE describing the evolution of cell density $u(x,t)$
\begin{equation}\label{01}
u_t  =  \nabla\cdot(u\nabla u-\chi u\nabla v) \,.
\end{equation}
The rigorous derivation for one single cell type from agent based models was done in \cite{Ol}, see also \cite{BCM,BV} and the references therein. Let us point out that there other ways of including volume effects such as the volume filling assumption \cite{HPVolume} differing from the volume exclusion considered here \cite{CCVolume}.  This basic model shows very rich dynamical properties and complex set of stationary states and metastability both for one species and multispecies cases \cite{BFH,CCH15,CHS17,BDFS,CCY} dealing with other attractive kernels instead of the classical chemotaxis kernels.  As usual in chemotaxis modeling, the previous equation is coupled with a reaction-diffusion equation of the chemoattractant $v(x,t)$ typically  created and degraded linearly as
\begin{equation}\label{02}
v_t=\Delta v-v+u \,.
\end{equation}
System \eqref{01}-\eqref{02} has been studied thoroughly in the case of linear diffusion for the cell density in two main settings: the whole space and the bounded domain case with no-flux boundary conditions for both cell density and chemoattractant, see for instance \cite{BDP06,BCM08,BCC12,BL,BCKKLL} for the full space case, \cite{Horstmann03,CLM,WX2013} and the references therein for Neumann boundary condition and the variants of (\ref{01})--(\ref{02}) in \cite{GW1999,LNT,NT1,NT2}.

However, finding stationary states and the asymptotic behavior of this system with nonlinear quadratic diffusion has been elusive. The difficulties being how to show confinement of the mass and how to characterize all the steady states of the system after the early work \cite{CCVolume} showing global uniform bounds on the cell density. In the whole space case and in one and two dimensions this has been clarified only recently by taking advantage of the gradient flow properties of this system, which has a particularly important Liapunov functional of which stationary states are critical points. In the whole space case, the existence of compactly supported radially decreasing global minimizers of the energy was proven in \cite{CCV}. Taking advantage of this variational structure, the authors in \cite{CHVY} were able to show that all stationary states in the whole space must be radially decreasing and compactly supported about their center of mass. In short, in two dimensions for the classical Keller-Segel model, all stationary states with the right regularity in \cite{CHVY} are given by single bumps. This property generalizes to any case in which the uniqueness of radial stationary solutions is proven. For instance, this is shown in \cite{kaib} for regular interaction kernels including the Bessel potential in one dimension.

Stationary states for this system in bounded domains can be fairly more complicated. Let us start by mentioning that even for the simpler aggregation-diffusion equation of the form
$$
u_t  = \nu\Delta u + \nabla\cdot(u\nabla(W*u)) \,,
$$
with $W$ being an interaction potential subject to periodic boundary conditions, this is a case that can demonstrate a phase transition depending on the strength of the noise $\nu$, see \cite{CGPS} and the references therein. These phenomena were also analyzed in \cite{CKY} in the case of quadratic diffusion showing sufficient conditions on the potential $W$ for phase transitions to happen.

In this paper, we consider the following one-dimensional chemotaxis model with quadratic cellular diffusion subject to Neumann boundary conditions
\begin{equation}\label{11}
\left\{\begin{array}{ll}
u_t=(u u_x-\chi u v_x)_x,&x\in(0,L),t>0,\\
v_t=v_{xx}-v+u, &x\in(0,L),t>0,\\
u(x,0),v(x,0)\geq 0, \mbox{ but }\not\equiv 0,&x\in(0,L),\\
u_x (x,t)=v_x (x,t)=0,& x=0,L, t>0.
\end{array}
\right.
\end{equation}
We shall show that system \eqref{11} exhibits a complicated phase transition phenomena due to the fine/intricate structures of its steady states that depend on the chemotactic sensitivity $\chi$.  To this end, we shall study its non-negative steady states, i.e., solutions to the following system
\begin{equation}\label{12}
\left\{\begin{array}{ll}
(uu_x -\chi uv_x)_x=0,&x\in(0,L),\\
v_{xx}-v+u=0,&x\in(0,L),\\
u_x =v_x =0,& x=0,L,\\
u(x)\geq0,v(x)>0,&x\in(0,L),\\[1mm]
(u,v)\in C^0(0,L)\times C^2(0,L), \int_0^L u(x)dx=M,
\end{array}
\right.
\end{equation}
and investigate how the structure and behavior of \eqref{12} change with respect to the chemotactic sensitivity parameter $\chi>0$.  We note that an immediate consequence of the zero-flux boundary conditions is the conservation of cell population
\[M=\int_0^L u(x,t)\,dx=\int_0^L  u(x,0)\,dx, \mbox{ for all } t>0,\]
which further implies that the constant pair $(\bar u,\bar v):=\Big(\frac{M}{L},\frac{M}{L}\Big)$ is a solution to (\ref{11}) and (\ref{12}).

The main aim of this work is to show in details the qualitative information encoded in Figure \ref{bifurcation}, which turns out to be closely related to the following parameter:
$$\chi_k=\Big(\frac{k \pi}{L}\Big)^2+1, \ \ k=1,2,3, \cdots.$$
Here, we see that the bifurcation diagram for steady states is quite intricate and vertical bifurcation occurs from the constant solution $(\bar u,\bar v)$ at each bar $\chi=\chi_k:=(\frac{k\pi}{L})^2+1$, $k\geq 1$.

Since the diffusion of the first equation of \eqref{12} is degenerate at $u=0$, one often expects that the solution component $u$ will have a compact support: the solution has a region with $u$ positive surrounded by vanishing regions.   Throughout the paper, we call the solution a ``bump'' if it is positive in some regions surrounded by two regions of vacuum, and ``half-bump'' is a bump cut in its middle.  We say a ``bump'' solution is similar if it is obtained by reflecting a ``half-bump' solution once or many times.  Then our main results can be summarized as follows:

\begin{theorem}\label{theorem11}
Let the cell mass $M>0$ in \eqref{12} be arbitrary. Then for each $\chi \in [\chi_1, \infty)\setminus \{\chi_k\}_{k=1}^{\infty}$, the solution of \eqref{12} must have a compact support in $(0,L)$ and it has at most $k$ half--bumps if $\chi<\chi_{k+1}$ with $k\geq 1$.  More specifically we have the following results:
\begin{enumerate}
\item  for $\chi<\chi_1$, (\ref{12}) has only the positive constant solution $(\bar u,\bar v)$, which is the global and exponential attractor of (\ref{11});
\item for each $\chi=\chi_k, k=1,2,\cdots$, there exists a one--parameter family of non-negative solutions to (\ref{12}) of the form
\begin{equation*}
(u_k^\epsilon(x),v_k^\epsilon(x))=(\bar u,\bar v)+\epsilon(\chi_k,1)\cos \frac{k\pi x}{L},  \quad \mbox{for all } \epsilon \in \bigg[-\frac{\bar u}{\chi_k},\frac{\bar u}{\chi_k}\bigg]
\end{equation*}
which are strictly positive in $[0,L]$ whenever $\epsilon \neq \pm \frac{\bar u}{\chi_k}$;

\item for each $\chi\in(\chi_1,\infty)$, (\ref{12}) admits a pair of half--bump solutions $(u,v)(x)$ and $(u,v)(L-x)$ explicitly given by:
\begin{equation*}
u(x)=\left\{\begin{array}{ll}
\mathcal A\big(\cos \omega x-\cos \omega l^*\big),&x\in(0,l^*),\\
0,&x\in(l^*,L),
\end{array}
\right.
 \ v(x)=\left\{\begin{array}{ll}
\mathcal A\big(\frac{\cos \omega x}{\chi}-\cos \omega l^*\big),&x\in(0,l^*),\\
\mathcal B \cosh (x-L),&x\in(l^*,L),
\end{array}
\right.
\end{equation*}
where  $l^*\in \big(\frac{\pi}{2\omega}, \frac{\pi}{\omega}\big) \subset (0,L)$ is uniquely determined by the algebraic equation
$
\frac{1}{\omega} \tan \omega l^*=\tanh(l^*-L)
$
and
$$\mathcal A=\frac{\bar u L}{\frac{1}{\omega}\sin \omega l^*-l^*\cos \omega l^*}, \ \ \mathcal B=\frac{\bar u L(\frac{1}{\chi}-1)}{\big(\frac{1}{\omega}\tan \omega l^*-l^*\big)\cosh (l^*-L)};
$$
moreover, the above pair are the unique nonconstant monotone solutions to (\ref{12}).  Furthermore, if $\chi \in (\chi_1, \chi_2)$, then the above pair are the unique nonconstant solutions to (\ref{12}).

\item for $\chi \in [\chi_2, \infty)\setminus \{\chi_k\}_{k=2}^{\infty}$, the following statements hold:
\begin{itemize}
\item system (\ref{12}) has a unique pair of similar--bump solutions $(u_k^\pm,v_k^\pm)(x)$ with $k$ half--bumps explicitly given by \eqref{similar-bump} if and only if $\chi>\chi_k$;
\item if $\chi>\chi_k$, system (\ref{12}) also admits similar--bump solutions $(u_m^\pm,v_m^\pm)(x)$ with $m$ half--bumps for each $m=2,\cdots, k$; moreover, it has infinitely many asymmetric multi--bump solutions $(u_m^\#,v_m^\#)$ that have $m$ half--bumps;
\end{itemize}
\end{enumerate}
\end{theorem}
From the above results, we see that the solution structure of \eqref{12} become increasingly rich and complex as $\chi$ expands.  The qualitative information of Theorem \ref{theorem11} is encoded in the bifurcation diagram shown in Figure \ref{bifurcation} where we plot the explicit vertical bifurcation branches and their global continuums out of the constant solution at $\chi_k:=(\frac{k\pi}{L})^2+1$, $k\geq 1$.  

\begin{figure}[ht!]
\centering
\includegraphics[width=\textwidth,height=3in]{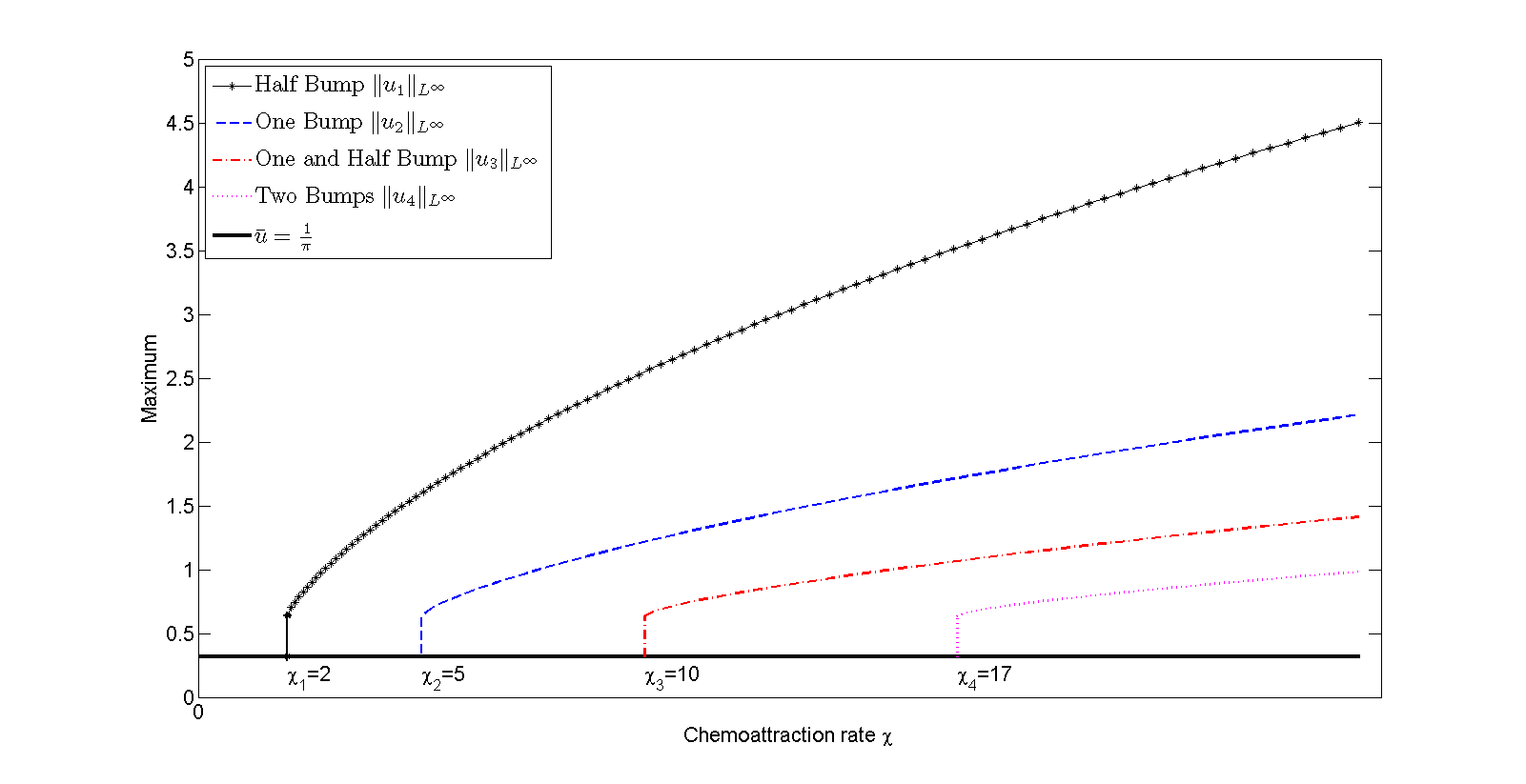}
  \caption{Bifurcation diagram of $L^\infty$-norm of $u$ vs. $\chi$ with $M=1$, $L=\pi$, and $\bar u=\bar v=\frac{1}{\pi}$.  Each vertical bar at $\chi_1=2,\chi_2=5,,\chi_3=10,\chi_4=17,...$ represent a bifurcation branch that consist of the one-parameter family of positive solutions $(u^\epsilon_k,v^\epsilon_k)$ given by result 2 of Theorem \ref{theorem11}.  For each $\chi\in\mathbb R$, the constant pair $(\bar u,\bar v)$ is always a solution of (\ref{12}), and it is
  globally asymptotically stable if $\chi<\chi_1$ and unstable if $\chi>\chi_1$.  For $\chi\in(\chi_1,\infty)$, (\ref{12}) admits half--bump solutions given by result 3 of Theorem \ref{theorem11}, unique up to a reflection about $x=\frac{L}{2}$; moreover, for $\chi\in(\chi_2,\infty)$, there exist solutions with double boundary bumps or a single interior bump.  In general, for $\chi\in(\chi_k,\infty)$, (\ref{12}) admits solutions with $k$ half--bumps, while for $\chi\in(\chi_k,\chi_{k+1})$, its solutions have at most $k$ half bumps.  For each $k$, we have that $\Vert u_k\Vert_{L^\infty}\rightarrow \infty$ as $\chi\rightarrow \infty$. For $\chi\geq\chi_2$ there are infinitely many asymmetric solutions to \eqref{12}.}
\label{bifurcation}
\end{figure}

Now that \eqref{12} admits more and more as $\chi>0$ increases, a question then naturally arises as to which solution or which type of solutions might be more stable than another.  This is very challenging and it is hardly possible to give a positive answer, however one may find some clues by comparing the size of energy at stationary solutions.  It is known that the system \eqref{12} admits the following dissipating energy
\begin{equation*}
\mathcal E(u,v)=\frac{1}{\chi}\int_0^L u^2dx+\int_0^L (v_x^2+v^2-2 uv)dx.
\end{equation*}
Though computing the energy of all possible solutions of \eqref{12} for large $\chi>0$ is impossible, another main result of this paper gives a complete hierarchy of all stationary symmetric bump-solutions as follows:
\begin{theorem}[Decay of energy of symmetric bump-solutions]
Assume that $\chi>\chi_k$, $k\geq1$, and let $(u_k(x),v_k(x))$ be the symmetric multi--bump solutions for $k\geq 2$, then their energies decays as the number of half--bumps increases
$$
\mathcal E(u_1,v_1)<\mathcal E(u_2,v_2)<...<\mathcal E(u_k,v_k)<\mathcal E(\bar u,\bar v).
$$
\end{theorem}

Section 3 is devoted to constructing and analyzing deeply the behavior of half-bumps, the compactly supported monotone distributional solutions, of \eqref{12} for $\chi>\chi_1$.  These half-bump solutions bifurcate from the limiting solutions of positive steady states $(u^\epsilon_1,v^\epsilon_1)$ as $\epsilon\rightarrow \pm \frac{\bar u}{\chi_1}$ at $\chi=\chi_1$ as depicted in Figure \ref{bifurcation}.  The same bifurcation occurs to the limiting solutions at $\chi=\chi_k$, $k\geq 2$, by a suitable reflection and gluing procedure. Moreover, asymmetric half-bumps and symmetric single bump solutions are also possible as soon as $\chi\geq \chi_2$. We point out that some of these families of solutions were already found in \cite{BCR}, and we revisit their analysis by complementing their results and understanding them in terms of the bifurcation diagram.  Moreover, the construction by symmetries of general branches is also novel with respect to \cite{BCR}.

Once this bifurcation diagram is analyzed in Section 4, we take advantage of the gradient flow structure to further analyze the stability of solutions from the energy landscape viewpoint. We show that the half-bump solutions are the ones with the least energy even if they become boundary spikes as $\chi\to\infty$. In the last section, we focus on the analysis of the stability of these branches via suitable numerical methods. Due to the gradient flow structure, we propose a structure-preserving numerical method in the spirit of \cite{CCH15} keeping the decreasing energy property of the system. Using this scheme we showcase the stability/instability of the different branches. We present certain conjectures about the basins of attraction of some branches and show the inherent metastability of solutions due to the large number of unstable stationary states and their complicated stable manifolds.


\section{Stability of Uniform Density and its Bifurcations}

In this section, we start by showing that the uniform steady state is exponentially asymptotically stable for all initial data whenever $\chi<\chi_1$ and we will show that a vertical bifurcation happens at every $\chi_k$, for all $k\geq 1$. Let us first point out that there are no nonconstant positive classical solutions to the steady state equation  (\ref{12}) on $[0,L]$ whenever $\chi\geq \chi_1$ unless $\chi=\chi_k$, $k\geq 1$.

\begin{lemma}\label{lemma21}
Let $(u,v)$ be an arbitrary non--constant solution of (\ref{12}). Then $u$ must be of compact support inside $[0,L]$ for each $\chi\in(\chi_1,\infty)$ unless $\chi=\chi_k$ for some $k\in\mathbb N^+$; for $\chi=\chi_k$,  there exists a one--parameter family of positive solutions to (\ref{12})
\begin{equation}\label{21}
(u_k^\epsilon(x),v_k^\epsilon(x))=(\bar u,\bar v)+\epsilon(\chi_k,1)\cos \frac{k\pi x}{L},  \quad \mbox{for all } \epsilon \in \Big(-\frac{\bar u}{\chi_k},\frac{\bar u}{\chi_k}\Big).
\end{equation} \end{lemma}

\begin{proof}
We first show that $u(x)$ to (\ref{12}) is of compact support on $[0,L]$ whenever $\chi>\chi_1$ unless $\chi=\chi_k$, $k\geq1$.  If not, assume that $\chi\in(\chi_1,\infty)\backslash \{\chi_k\}_{k=2}^\infty$ and $u(x)>0$ in $(0,L)$, then $u-\chi v$ equals a constant in $(0,L)$, while integrating it over $(0,L)$ implies that $u-\bar u=\chi(v-\bar v)$.  Then the $v$--equation becomes
\[
\left\{\begin{array}{ll}
(v-\bar v)_{xx}+(\chi-1)(v-\bar v)=0,&x\in(0,L),\\
v_x =0,& x=0,L,\\
\end{array}
\right.
\]
which has no nonconstant solution unless $\chi-1=(\frac{k\pi}{L})^2$ for some $k\in\mathbb N^+$, i.e., $\chi=\chi_k$.  Therefore, for $\chi\in(\chi_1,\infty)\backslash \{\chi_k\}_{k=2}^\infty$, we deduce that $v$ is constant implying $u$ is constant on $(0,L)$. Moreover, when $\chi=\chi_k$, $k\geq 1$, one can solve the equation explicitly to obtain this family of positive solutions \eqref{21} .
\end{proof}

One can also easily find that the constant solution $(\bar u,\bar v)$ of (\ref{11}) is locally stable if $\chi<\chi_1$ and is linearly unstable if $\chi>\chi_1$ by linearizing around the constant solution.
Let us show further that $(\bar u,\bar v)$ is globally asymptotically stable and (\ref{12}) has only the constant solution for $0<\chi<\chi_1$. We shall make use of several inequalities which are summarized in the following lemma.
\begin{lemma}
Let $\lambda_1$ be the principal eigenvalue for the Neumann problem of the $-\Delta$ operator in the domain $\Omega\subset\mathbb R^N$.  Assume that $u\in L^1_+\cap L^\infty (\Omega)$, $u\in H^1(\Omega)$, and $\int_\Omega u dx=\bar u |\Omega|$. Then the following inequality holds
\begin{equation}\label{QW1}
     \frac{1}{2\Vert u\Vert_{L^\infty}}\int_\Omega |u-\bar u|^2\leq \int_\Omega u\ln\Big(\frac{u}{\bar u}\Big)\leq \frac{2}{\bar u} \int_\Omega |u-\bar u|^2\leq \frac{2}{\bar u}\frac{1}{\lambda_1}  \int_\Omega |\nabla u|^2.
\end{equation}
\end{lemma}
\begin{proof}
The last inequality readily follows from Rayleigh quotient.  To show the first two, we introduce the following function $f(z):=z\ln \frac{z}{\bar u}+\bar u-z,z\geq0.$
One finds that $f(0)=\bar u$, $f(\bar u)=f'(\bar u)=0$ and $f''(z)=\frac{1}{z}$. Taylor expansion implies that
\[
f(u)=f(\bar u)+f'(\bar u)(u-\bar u)+\frac{f''(\xi)}{2}(u-\bar u)^2=\frac{1}{2\xi}(u-\bar u)^2\geq \frac{1}{2\Vert u\Vert_{L^\infty}}(u-\bar u)^2,
\]
since $z\in [u,\bar u]$, which gives rise to the first inequality. To show the second inequality take $\alpha\in (0,1)$ be a constant to be determined, then we have that
$f(z)=\frac{1}{2\xi}(z-\bar u)^2\leq \frac{(z-\bar u)^2}{2\alpha \bar u}, z\geq \alpha \bar u;$
on the other hand, since $f'(z)<0$ for $z<\bar u$, we have that
$f(z)\leq f(0)=\bar u\leq \frac{(z-\bar u)^2}{(1-\alpha)^2 \bar u}, 0\leq z \leq \alpha \bar u.$
Choosing $\alpha =2-\sqrt 3$ gives us the second inequality.
\end{proof}

We note that (\ref{11}) is globally well--posed in general bounded domains $\Omega\subset\mathbb R^N$ ($N\geq1$), and their weak solutions \cite{CLM} are uniformly bounded by and unique in the class of bounded weak solutions \cite{CLM}.  We can conclude the exponential convergence.

\begin{theorem}
Let $\Omega\subset\mathbb R^N$ ($N\geq1$).  If $0<\chi<\chi_1$, then for any non--negative initial data $(u_0,v_0)\in L^\infty(\Omega)^\times L^\infty(\Omega)$, the constant solution $(\bar u,\bar v)$ is the exponential global attractor of (\ref{11}), i.e., for any $1\leq p<\infty$, there exist two constants $C,\delta\in\mathbb R^+$ such that
\[\Vert u(\cdot,t)-\bar u \Vert_{L^p(\Omega)}+\Vert v(\cdot,t)-\bar v\Vert_{H^1(\Omega)}\leq Ce^{-\delta t}, t\geq0.\]
\end{theorem}
\begin{proof}
Let us introduce the following functional
\[\mathcal F:=\int_\Omega u\ln\Big(\frac{u}{\bar u}\Big)+\frac{\chi}{2}|\nabla v|^2.\]
Using \eqref{11} one can easily obtain
\begin{align}\label{QW2}
\frac{d\mathcal F}{dt} =&\int_\Omega (\ln u+1)u_t+\chi \nabla v\cdot \nabla v_t =-\int_\Omega |\nabla (u-\chi v)|^2-\chi \int_\Omega \Big(|\Delta v|^2+(1-\chi)|\nabla v|^2\Big).
\end{align}
To proceed, we claim from the Rayleigh quotient that $\int_\Omega |\Delta v|^2\geq \lambda_1 \int_\Omega|\nabla v|^2$, where $\lambda_1=(\frac{\pi}{L})^2$ in $(0,L)$ and it can be generalized to the principal Neumann eigenvalue in higher dimensions.  We present a simple proof of the claim for completeness.  Let $\{(\lambda_i,\psi_i)\}_{i=0}^\infty$ be the Neumann eigen--pairs, with $\Vert \psi_i\Vert_{L^2}=1$, then the eigen--expansion of $v=\sum C_i\psi_i$ implies
\[\int_\Omega |\nabla v|^2=-\int_\Omega v \Delta v =\sum \lambda_i C_i^2, \int_\Omega |\Delta v|^2=\sum \lambda_i^2 C_i^2,\]
then we have that
\[\int_\Omega |\Delta v|^2-\lambda_1\int_\Omega |\nabla v|^2=\sum \lambda_i(\lambda_i-\lambda_1)C_i^2\geq0,\]
from which the claim follows.  Then one can further proceed with $\chi_1=1+\lambda_1$ to get from (\ref{QW2}) that
\begin{align}\label{QW3}
\frac{d\mathcal F}{dt} \leq &-\int_\Omega |\nabla (u-\chi v)|^2-\chi (\chi_1-\chi) \int_\Omega |\nabla v|^2\nonumber\\
=&-(1-\epsilon)\int_\Omega |\nabla (u-\chi v)|^2- \epsilon\int_\Omega |\nabla (u-\chi v)|^2-\chi (\chi_1-\chi) \int_\Omega |\nabla v|^2\nonumber\\
\leq&-\frac{\epsilon}{2} \int_\Omega |\nabla u|^2-\chi\Big(\chi_1-(1+\epsilon)\chi\Big)\int_\Omega|\nabla v|^2,
\end{align}
where $\epsilon>0$ is small and the last inequality follows from the pointwise inequality $|\nabla u-\chi\nabla v|^2\geq \frac{|\nabla u|^2}{2}-\chi^2|\nabla v|^2$.
In light of the second inequality in (\ref{QW1}), we have from (\ref{QW3}) that
\[\frac{d\mathcal F}{dt} \leq -\frac{\lambda\epsilon\bar u}{4} \int_\Omega u\ln\Big(\frac{u}{\bar u}\Big)-\chi\Big(\chi_1-(1+\epsilon)\chi\Big)\int_\Omega|\nabla v|^2\leq -\delta\mathcal F,\]
with $\delta:=\min\Big\{\frac{\lambda \epsilon \bar u}{4},2\big(\chi_1-(1+\epsilon)\chi\big)\Big\}$.  This implies that $\mathcal F$ converges to zero exponentially as $t\rightarrow\infty$, which yields the exponential convergence of $\|u-\bar{u}\|_{L^2(\Omega)}$ due to \eqref{QW1}. Since $u$ is uniform-in-time bounded by classical arguments, see \cite{Kowalczyk04,CCVolume,TaoWinkler12},  we have that $\|u-\bar{u}\|_{L^\infty(\Omega)}$ decays exponentially by $L^p$-interpolation. Furthermore using the parabolic regularity to the second equation of \eqref{11}, we get that $\Vert v-\bar v\Vert_{H^1(\Omega)}$ also decays exponentially. This completes the proof.
\end{proof}

\begin{figure}[ht]
        \centering
\includegraphics[width=\textwidth,height=3in]{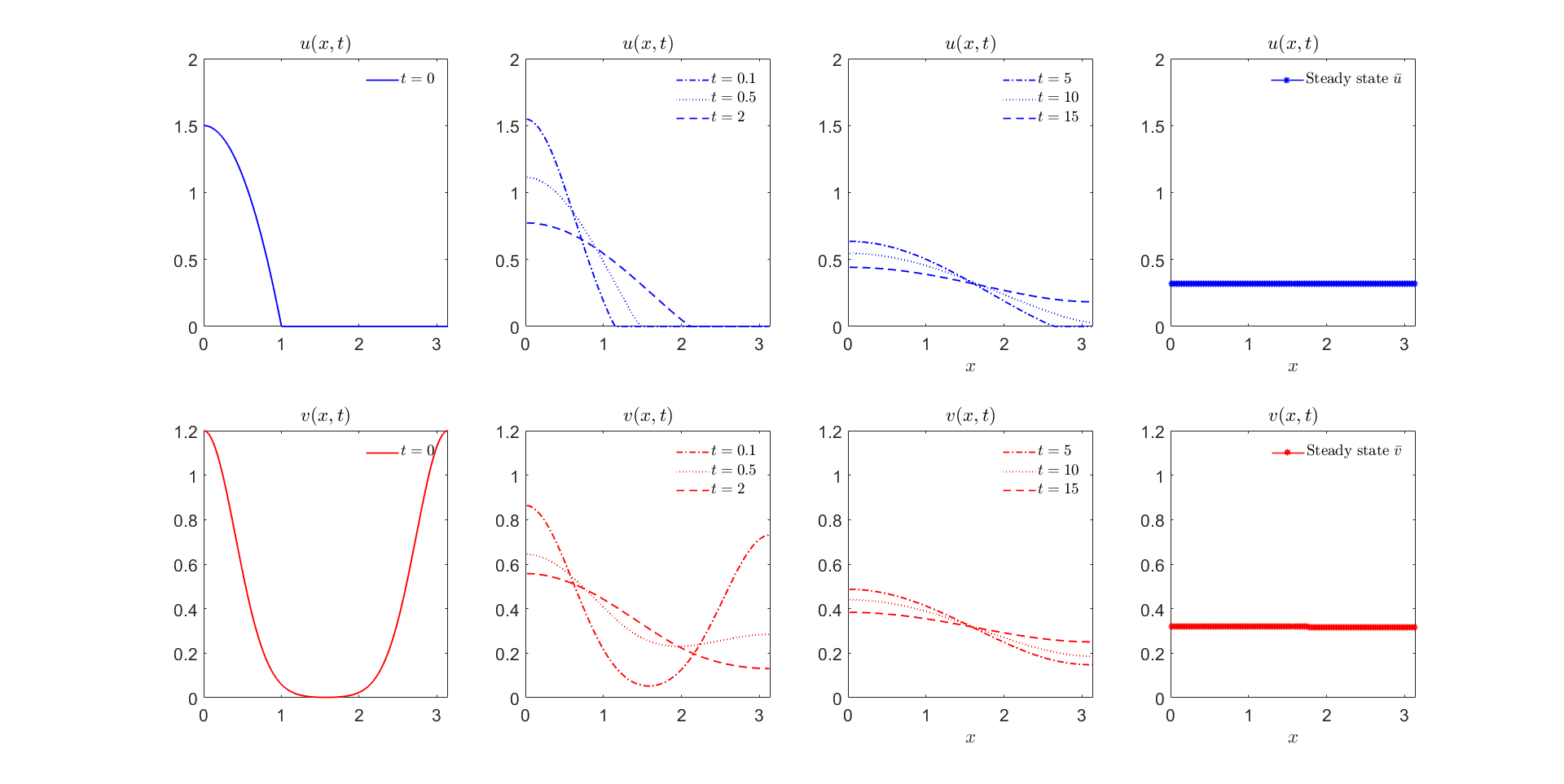}
  \caption{Convergence to the constant solution $(\bar u,\bar v)=(\frac{1}{\pi},\frac{1}{\pi})$ out of initial data $u_0(x)=\max\{0,\frac{3}{2}\big(1-x^2\big)\}$ and $v_0(x)=1.2e^{-3x^2}+1.2e^{-3(x-\pi)^2}$ for $\chi=1.5< \chi_1=2$ over $(0,\pi)$.  This illustrates our theoretical result that $(\bar u,\bar v)$ is the global attractor if $\chi<\chi_1$.}\label{QWfigure1}
\end{figure}

It seems necessary to point out that, $(u_k,v_k)$ given by (\ref{21}) are solutions that bifurcate from $(\bar u,\bar v)$ at $\chi=\chi_k$, $k\in\mathbb N^+$.  Indeed, one can apply the local bifurcation theory of Crandall and Rabinowitz \cite{CR} and its user--friendly development \cite{SW2009,WX2013} as follows: let us denote
$
\mathcal{X}=\{w \in H^2(0,L), w>0, \vert w'(0)=w'(L)=0\},
$
with $'=\frac{d}{dx}$. Then we rewrite (\ref{12}) into the abstract form by taking $\chi$ as the bifurcation parameter
$\mathcal{F}(u,v,\chi)=0,~(u,v,\chi) \in \mathcal{X}  \times \mathcal{X} \times \mathbb{R},$
where
\begin{equation*}
\mathcal{F}(u,v,\chi) =\left(
 \begin{array}{c}
(uu'-\chi u v')'\\
v''-v+u
 \end{array}
 \right).
 \end{equation*}
It is easy to see that $\mathcal{F}(\bar{u},\bar{v},\chi)=0$ for any $\chi \in \mathbb{R}$ and that $\mathcal{F}: \mathcal{X}  \times \mathbb{R}  \times
\mathbb{R}  \rightarrow \mathcal{Y} \times \mathcal{Y}$ is analytic for $\mathcal{Y}=L^2(0,L)$. Moreover, one can find through straightforward calculations
that, at any fixed $(u_0,v_0) \in \mathcal{X} \times \mathcal{X}$, the Fr\'echet derivative $D_{(u,v)}\mathcal{F}(u_0,v_0,\chi)$ is a Fredholm operator with zero index. Furthermore, the null space $\mathcal{N}\big(D_{(u,v)}\mathcal{F}(\bar{u},\bar{v},\chi)\big)$ is not empty if and only if $\chi=\chi_k$, and the null space $\mathcal{N}(D_{(u,v)}\mathcal{F}(\bar{u},\bar{v},\chi_k))$ is of one dimension and has a span
\[\mathcal{N}(D_{(u,v)}\mathcal{F}(\bar{u},\bar{v},\chi_k))=\text{span}\left\{ (\chi_k,1)\cos\frac{k\pi x}{L} \right\}, k\in\mathbb N^+.\]
Then for each $k\in \mathbb N^+$, there exists a (small) constant $\delta>0$ such that the following analytic functions
$s\in(-\delta, \delta):\rightarrow (u_k(s,x),v_k(s,x),\chi_k(s)) \in \mathcal{X}\times \mathcal{X} \times \mathbb{R}^+$
\begin{equation}\label{24}
\left\{
\begin{array}{ll}
u_k(s,x)=\bar{u}+s\chi_k\cos \frac{k\pi x}{L}+s^2\varphi_1(x)+s^3\varphi_2(x)+\mathcal O(s^4),\\
v_k(s,x)=\bar{v}+s\cos \frac{k\pi x}{L}+s^2\psi_1(x)+s^3\psi_2(x)+\mathcal O(s^4),\\
\chi_k(s)=\chi_k+sK_1+s^2K_2+O(s^3),
\end{array}
\right.
\end{equation}
solve the system (\ref{12}), where $\mathcal O(s^4)$ is defined with $\mathcal{X}$-topology, $K_i (i=1,2)$ are constants and $(\varphi_i,\psi_i)\in \mathcal{Z}$ with
\begin{equation*}
\mathcal{Z}=\bigg\{(\varphi,\psi)\in \mathcal{X} \times \mathcal{X}~ \big \vert \int_0^L (\chi_k\varphi+\psi)\cos\frac{k\pi x}{L}dx=0\bigg\}.
\end{equation*}
Moreover, all nontrivial solutions of (\ref{12}) near the bifurcation point ($\bar{u},\bar{v}, \chi_k)$ lie on the curve $\Gamma_k(s)=(u_k(s), v_k(s), \chi_k(s))$, $s\in(-\delta,\delta)$.

\begin{figure}[ht]
        \centering
\includegraphics[width=\textwidth,height=3in]{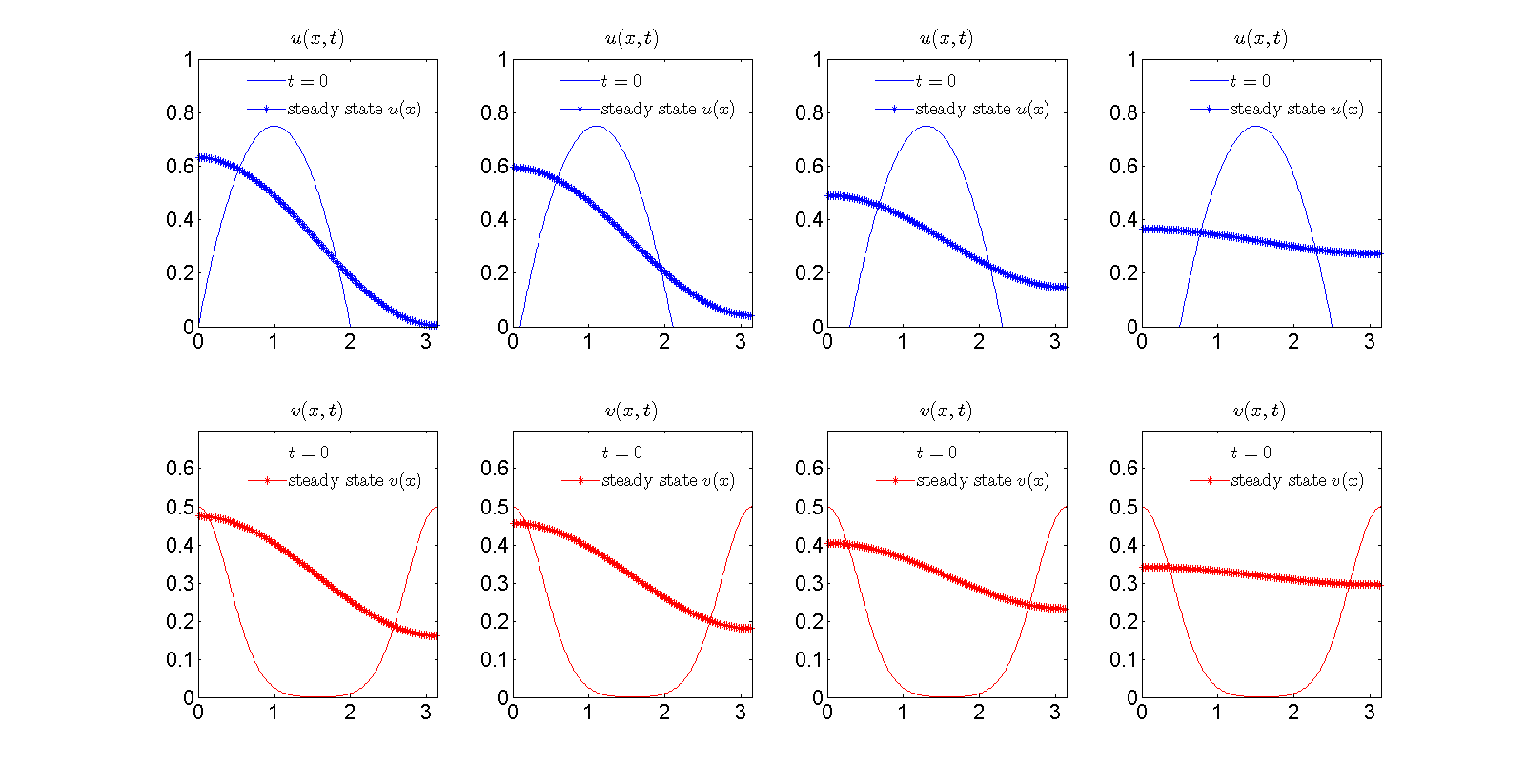}
  \caption{Several positive steady states of (\ref{11}) in $(0,\pi)$ given by (\ref{21}) are achieved when $\chi=\chi_1=2$ subject to different initial data.  According to Lemma \ref{lemma21}, any solution of (\ref{12}) must take the form of (\ref{21}) whenever $\chi=\chi_1=2$, while this one-parameter family of solutions have the same energy.  The initial data are chosen to be $u_0(x)=0.75\max\{0,1-(x-a)^2\}$ and $v_0(x)=0.5e^{-3x^2}+0.5e^{-3(x-L)^2}$, with $a=1,1.1,1.3$ and $1.5$ from the left to right column.  Our numerical results suggest that the dynamical system \eqref{11} with $\chi=\chi_1$ can stabilize into both strictly positive steady states or those touching zero at the boundary point even if the initial datum $u_0(x)$ is compactly supported.} \label{QWfigure2}
\end{figure}

One can compare (\ref{21}) with (\ref{24}) to speculate that $\varphi_i,\psi_i\equiv0 $, $K_i=0$, and so are for all the higher order terms. Therefore, all solutions around $(\bar u,\bar v,\chi_k)$ must be of the form given by (\ref{21}).  To see it in full details, we substitute (\ref{24}) into (\ref{12}) and arrange the equation in the order of $s$.  First of all, collecting $s$--terms easily gives us $K_1=0$; moreover, collecting $s^2$--terms gives
 \[
\left\{\begin{array}{ll}
\varphi_1-\chi_k\psi_1=0, &x\in(0,L)\\
\psi_1''-\psi_1+\varphi_1=0,&x\in(0,L).
\end{array}
\right.
\]
We solve these two equations and find that $\varphi_1=\psi_1\equiv 0$ for all $x\in(0,L)$.  Furthermore, we collect the $s^3$--terms to find that
 \[
\left\{\begin{array}{ll}
(\varphi_2-\chi_k\psi_2)'=K_2(\cos\frac{k\pi x}{L})', &x\in(0,L)\\
\psi_2''-\psi_2+\varphi_2=0,&x\in(0,L).
\end{array}
\right.
\]
We test the first equation above against $\sin \frac{k\pi x}{L}$ and find that $K_2=0$.  Moreover, testing the second equation against $\cos\frac{k\pi x}{L}$ gives us, thanks to the fact that $(\varphi_2,\psi_2)\in\mathcal Z$, that
\[\int_0^L \varphi_2 \cos\frac{k\pi x}{L}dx=\int_0^L \psi_2 \cos\frac{k\pi x}{L}dx=0.\]
Then one can show that $\varphi_2=\psi_2\equiv 0$.  Similarly, we can show that all the rest terms are zeros.  Therefore, our claim is verified. These results are illustrated in Figures \ref{QWfigure1} and \ref{QWfigure2}.  Though, the degeneracy at $u=0$ refrains us from applying the global bifurcation as in \cite{SW2009,WX2013}, we are able to know well about the global structures of each branch as in Figure \ref{figure1} thanks to the explicit solutions.


\section{Bifurcating Bump Solutions}

In this section, we will first construct explicit compactly supported weak solutions to (\ref{12}) which are decreasing monotonically in their support called half-bumps and study their properties with respect to the chemoattractant sensitivity. Moreover, we will see that this thorough study of the half-bump solutions allows for the easy construction of a family of multi-bump solutions by reflection. This family of solutions will be referred to as the similar multi-bump family of solutions since all are constructed from the basic brick of a single half-bump solution. Finally, we will construct asymmetric multi-bump solutions showing how intricate the bifurcation diagram of this problem can be. Let us remind that, as already mentioned in the introduction, some of these results were partly obtained in \cite{BCR}. However, we revisit them here by constructing them in a different manner, expanding the analysis of their properties, and putting them in context with the bifurcation diagram in Figure \ref{bifurcation}.

\subsection{Half-bumps}

We first study monotone decreasing solutions of (\ref{12}) with a compact support for $\chi>\chi_1$, i.e., $u(x)>0$ for $x\in [0,l^*)$ and $u(x)\equiv 0$ for $x\in[l^*,L]$ with $l^*$ to be determined.  Then we shall construct its solutions with multi--bumps by building on the reflecting and extending this block.  Denote in the sequel
$
\omega:=\sqrt{\chi-1}.
$
First of all, since $u>0$ in $[0,l^*)$, we can readily see from the $u$--equation that there exists some constant $\lambda $ to be determined such that
$u-\chi v=\lambda, x\in(0,l^*)$,
hence the $v$--equation becomes
\[
\left\{\begin{array}{ll}
v_{xx}+(\chi-1)v+\lambda =0,&x\in(0,l^*),\\
v_{xx}-v=0,&x\in(l^*,L),\\
v_x(0)=v_x(L) =0.\\
\end{array}
\right.
\]
Solving the above problem directly gives us that
\[
v(x)=\left\{\begin{array}{ll}
C_1\cos \omega x-\frac{\lambda }{\chi-1},&x\in(0,l^*),\\
C_2\cosh (x-L),&x\in(l^*,L),
\end{array}
\right.
\]
for some constants $C_1$ and $C_2$ to be determined. It can be easily checked that $v'(x)<0$ for $x\in (0,L)$.
By the continuity of both $v'(x)$ and $v''(x)$ at $x=l^*$ we have
\[
\left\{\begin{array}{ll}
-C_1\omega\sin \omega l^*=C_2 \sinh(l^*-L),\\
-C_1\omega^2\cos \omega l^*=C_2 \cosh(l^*-L),
\end{array}
\right.
\]
which entail that $l^*$ is a root of the algebraic equation
\begin{equation}\label{27}
\frac{1}{\omega} \tan \omega l^*=\tanh(l^*-L).
\end{equation}

For each $\chi>\chi_1$, one can easily show that there exists a unique $l^*\in \big(\frac{\pi}{2\omega}, \frac{\pi}{\omega}\big) \subset (0,L)$ that solves (\ref{27}).  Indeed, let us denote
\[f(\xi;\omega):=\frac{1}{\omega}\tan \omega \xi-\tanh (\xi-L), \xi\in(0,L).\]
Then $\omega>\frac{\pi}{L}$ for $\chi>\chi_1$, hence $f(\frac{\pi}{\omega};\omega)=\tanh (L-\frac{\pi}{\omega})>0$; moreover,  $f((\frac{\pi}{2\omega})^+;\omega)=-\infty<0$ and $f(\xi;\omega)\in C^\infty((\frac{\pi}{2\omega},\frac{\pi}{\omega}))$ imply that $f(l^*;\omega)=0$ for some $l^*\in(\frac{\pi}{\omega},\frac{\pi}{2\omega})$, while $f_\xi(\xi;\omega)=\tan^2 \omega \xi+\tanh^2(\xi-L)>0$ ensures that $l^*$ is unique.  We would like to point out that for $\chi<\chi_1$, $f(\xi;\omega)$ admits no positive root and hence (\ref{12}) has only constant solution $(\bar u,\bar v)$; moreover, for $\chi\in(\chi_1,\chi_2]$ its has a unique solution $l^*\in(\frac{\pi}{\omega},\frac{\pi}{2\omega})$.  It is also necessary to point out that when $\chi>\chi_2$, $f(\xi;\chi)$ has multiple roots, at least $l^{**}\in (\frac{3\pi}{2\omega},\frac{2\pi}{\omega})$.  However, all the rest roots will be ruled out since we look for non--negative solutions $(u,v)$.  Indeed, if not, $u(l^{**})=0$ and $u(x)=C(\cos \omega x-\cos \omega l^{**})$ for $x\in (0,l^{**})$, then we have that $u(x)<0$ for $x\in(\frac{\pi}{\omega}, \frac{3 \pi}{2\omega})$, which is apparently not physical.  Similarly, one can show that all the other roots of $f(\xi;\omega)=0$, if exist at all, are not applicable.  Therefore $l^*\in (\frac{\pi}{2\omega},\frac{\pi}{\omega})$ is the unique root that we look for as claimed above.

With $l^*$ being obtained through (\ref{27}), we find that
\begin{equation}\label{28}
u(x)=\left\{\begin{array}{ll}
\mathcal A\big(\cos \omega x-\cos \omega l^*\big),&x\in(0,l^*),\\
0,&x\in(l^*,L),
\end{array}
\right.
\end{equation}
where the conservation of total population $\int_0^{l^*} u(x)dx=\mathcal A\big(\frac{1}{\omega}\sin \omega x-x\cos \omega l^* \big)\big|_0^{l^*} =\bar u L$ gives
\begin{equation}\label{29}
\mathcal A=\frac{\bar u L}{\frac{1}{\omega}\sin \omega l^*-l^*\cos \omega l^*}.
\end{equation}
Note that $l^*\in(\frac{\pi}{2\omega},\frac{\pi}{\omega})$, hence $\omega l^* \in (\frac{\pi}{2},\pi)$ and it follows that $\mathcal A>0$.

To find $v(x)$, we have from $u=\chi v+\bar v$ for $x\in(0,l^*)$ that
$\mathcal A\big(\cos \omega x-\cos \omega l^*\big)=\chi C_1 \cos \omega x-\frac{\lambda }{\omega^2},$ which implies $C_1=\frac{\mathcal A}{\chi}$ and $\lambda =\mathcal A(\chi-1)\cos \omega l^*$, i.e.,
\begin{equation*}
\lambda =\frac{\bar u L \omega^2}{\frac{1}{\omega}\tan \omega l^*-l^*}.
\end{equation*}
Hence we find that $v(x)=\mathcal A\big(\frac{\cos \omega x}{\chi}-\cos \omega l^*\big)$ in $(0,l^*)$; on the other hand, by the continuity of $v(x)$ at $x=l^*$, we equate $\mathcal (\frac{1}{\chi}-1) \cos \omega l^*=\mathcal B \cosh (l^*-L)$ to obtain
\begin{equation}\label{211}
\mathcal B=\frac{\bar u L(\frac{1}{\chi}-1)}{\big(\frac{1}{\omega}\tan \omega l^*-l^*\big)\cosh (l^*-L)},
\end{equation}
with $l^*$ obtained in (\ref{27}).  Collecting the above calculations, we find that
\begin{equation}\label{212}
v(x)=\left\{\begin{array}{ll}
\mathcal A\big(\frac{\cos \omega x}{\chi}-\cos \omega l^*\big),&x\in(0,l^*),\\
\mathcal B \cosh (x-L),&x\in(l^*,L),
\end{array}
\right.
\end{equation}
where $\mathcal A$ and $\mathcal B$ are given by (\ref{29}) and (\ref{211}).

According to our discussion above, (\ref{12}) admits such half--bump solutions for each $\chi>\chi_1$.  Moreover, as we shall show in the coming section, it also has multi--bump solutions for $\chi>\chi_2$, however $l^*$ is unique whenever $\chi\in(\chi_1,\chi_2]$ and this indicates that the non--constant half--bump solution to (\ref{12}) is unique.  The following results can be summarized.
\begin{proposition}Let $M>0$ be an arbitrary constant in (\ref{12}), then:

(i). if $\chi<\chi_1$, (\ref{12}) has only the constant solution $(\bar u,\bar v)$;

(ii). if $\chi=\chi_1$, the solution of (\ref{12}) must take the form of $(u_1^\epsilon,v_1^\epsilon)$ in (\ref{21});

(iii). for each $\chi>\chi_1$, (\ref{12}) has a pair of half--bump solutions $(u,v)(x)$ and $(u,v)(L-x)$, explicitly given by (\ref{28}) and (\ref{212});

(iv). if $\chi\in(\chi_1,\chi_2)$, the nonconstant solution of (\ref{12}) must be the half--bumps given in (iii).
\end{proposition}

\begin{remark}
Notice that the results (ii) and (iii) in the previous proposition are obtained in \cite{BCR} with different arguments, while (i) and (iv) are new results. We point out that the half-bump solutions can be intuitively thought as a bifurcating curve whose starting point is the limit of the positive solutions on the interval $[0,L]$ for $\chi=\chi_1$ touching zero in one side.
\end{remark}

In Figure \ref{figure1}, we plot $(u(x),v(x))$ given by (\ref{28}) and (\ref{212}) and its reflection $(u(L-x),v(L-x))$..
\begin{figure}[ht]
        \centering
\includegraphics[width=5in,height=1.5in]{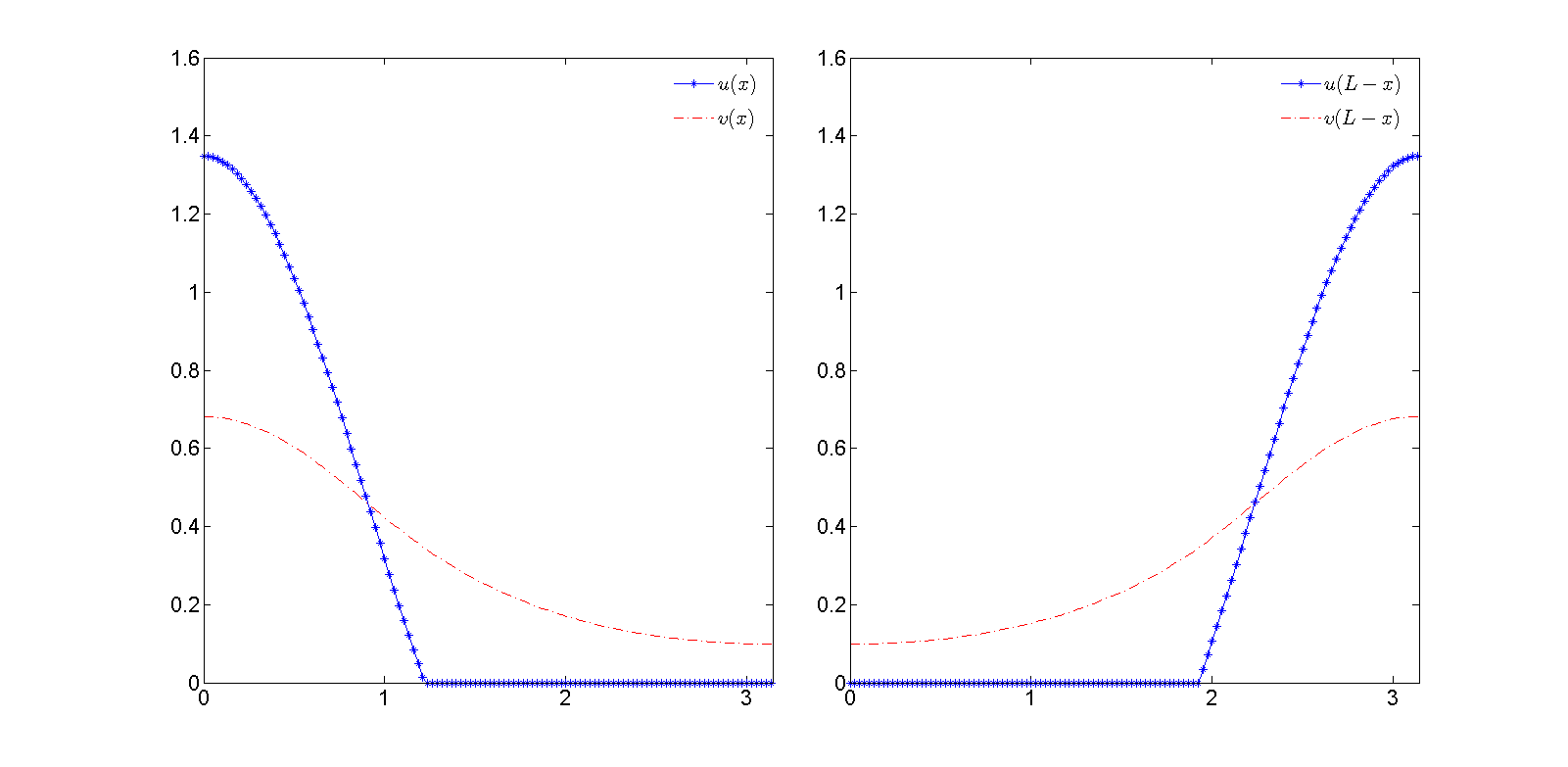}
  \caption{Half--bump solution $(u,v)(x)$ of (\ref{12}) from (\ref{28}) and (\ref{212}) and its reflection about $\frac{L}{2}$ in $(0,\pi)$.  A unit total cell population and $\chi=4$ are chosen, and the size of support $l^*=1.22$.}\label{figure1}
\end{figure}

\subsection{Asymptotic behavior of half-bumps in the limit of $\chi\rightarrow \infty$}

Next we study the effect of large $\chi$ on the half--bump $(u,v)$ established in (\ref{28}) and (\ref{212}).  In particular, we shall show that $u$ converges to a Dirac--delta function as $\chi$ goes to infinity.  First of all, we note from above that the size $l^*\in(\frac{\pi}{2\omega},\frac{\pi}{\omega})$ of support of $u(x)$ is uniquely determined by and continuously depends on $\chi$.  We claim that $l^*$ is strictly decreasing in $\chi$, and it is equivalent to show that $\frac{\partial l^*}{\partial \omega}<0$ for $\omega>0$.

Differentiating both sides of (\ref{27}) with respect to $\omega$, we have
\[-\frac{1}{\omega}\tan \omega l^*+\frac{1}{\omega \cos^2 \omega l^*}\Big(l^*+\omega\frac{\partial l^*}{\partial \omega}\Big)= \frac{1}{\cosh^2(l^*-L)}\frac{\partial l^*}{\partial \omega}.\]
With the identities $\frac{1}{\cos^2 \omega l^*}=1+\tan^2 \omega l^*$ and $\frac{1}{\cosh^2(l^*-L)}=1-\tanh^2(l^*-L)$, one gets
\[\frac{\partial l^*}{\partial \omega}=\frac{-l^*\tan^2\omega l^*+\frac{1}{\omega}\tan \omega l^*-l^*}{\omega(\tan^2\omega l^*+\tanh^2(l^*-L))},\]
then it is easy to find that $\frac{\partial l^*}{\partial \omega}<0$ since the numerator is strictly negative thanks to the fact $\omega l^* \in (\frac{\pi}{2}, \pi)$.

We next show that the magnitude of $u(x)$ is strictly increasing in $\chi$ such that $\frac{\partial \Vert u\Vert_{L^\infty}}{\partial w}>0$ for each $\chi>\chi_1$ and $\Vert u\Vert_{L^\infty}=\omega+o(1)$ as $\chi\rightarrow \infty$.  To prove the former, we denote
$z=\omega l^*$ and rewrite $\Vert u\Vert_{L^\infty}=\frac{M\omega (1-\cos z)}{\sin z-z\cos z}$. Then we have
\begin{align}
\frac{\partial \Vert u\Vert_{L^\infty}}{\partial \omega}
=&\frac{(1-\cos z +\omega \sin z \frac{\partial z}{\partial \omega})(\sin z-z\cos z)-w(1-\cos z) z\sin z \frac{\partial z}{\partial \omega}}{(\sin z-z\cos z)^2} \nonumber \\
=&\frac{(1-\cos z)(\sin z-z\cos z)+\omega \sin z (\sin z-z)\frac{\partial z}{\partial \omega}}{(\sin z-z\cos z)^2},  \nonumber
\end{align}
which, in light of the identity $\frac{\partial z}{\partial\omega }=l^*+\omega \frac{\partial l^*}{\partial\omega }$, becomes
\begin{align}
\frac{\partial \Vert u\Vert_{L^\infty}}{\partial \omega}
=&\frac{(1-\cos z)(\sin z-z\cos z)+z \sin z (\sin z-z)}{(\sin z-z\cos z)^2}+\frac{w^2\sin z (\sin z-z)}{(\sin z-z\cos z)^2}\frac{\partial l^*}{\partial \omega} \nonumber\\
\geq &\frac{(1-\cos z)(\sin z-z\cos z)+z \sin z (\sin z-z)}{(\sin z-z\cos z)^2},\nonumber
\end{align}
where we have applied the fact that $\sin z\leq z$, $z\in(\frac{\pi}{2},\pi)$ and $\frac{\partial l^*}{\partial \omega}<0$ for the inequality.

Denote
$g(z):=(1-\cos z)(\sin z-z\cos z)+z \sin z (\sin z-z).$
Then, in order to prove $\frac{\partial \Vert u\Vert_{L^\infty}}{\partial \omega}>0$, it  suffices to show that $g(z)>0$, $z\in(\frac{\pi}{2},\pi)$.  To this end, we first observe that $g(\frac{\pi}{2})=1+\frac{\pi}{2}-(\frac{\pi}{2})^2>0$; moreover, we find $g'(\frac{\pi}{2})=2+(\frac{\pi}{2})^2-\frac{\pi}{2}>0$ and
$g''(z)=(4\sin z -3z)\cos z +(z^2-1)\sin z>0$,
therefore $g'(z)>0$ and hence $g(z)>g(\frac{\pi}{2})>0$ for all $z\in (\frac{\pi}{2},\pi)$ as expected.  This finishes the proof.

To verify the latter, we first claim that $z:=\omega l^*\rightarrow (\frac{\pi}{2})^+$ as $\chi\rightarrow \infty$.  If not, say $\omega l^*\rightarrow \theta\in (\frac{\pi}{2},\pi]$ as $\chi \to \infty$, then the left hand side of (\ref{27}) converges to zero and as a consequence $l^*\rightarrow L^-$ as $\chi\rightarrow \infty$. However yields a contradiction since $l^*\in (\frac{\pi}{2\omega},\frac{\pi}{\omega})$ and $l^*\rightarrow 0^+$ as $\chi\rightarrow \infty$.  Then it is easy to see that $\Vert u\Vert_{L^\infty}=\frac{\omega (1-\cos z)}{\sin z-z\cos z}=\omega+o(1)$ as $\chi\rightarrow \infty$.

We proceed to study the effect of large $\chi$ on the profiles of $u(x)$ and $v(x)$ given by (\ref{28}) and (\ref{212}).  Since $l^*\rightarrow 0^+$ and $\omega l^*\rightarrow (\frac{\pi}{2})^+$ as $\chi\rightarrow \infty$, it is easy to see that $\mathcal A\rightarrow \infty$ and we have that $u(x)\rightarrow u_\infty=\bar u L \delta(x)$ pointwisely, where $\delta(x)$ is the Dirac Delta at $x=0$.

On the other hand, as $\chi\rightarrow \infty$, $v(x)\rightarrow v_\infty(x)$ pointwisely in $(0,L)$, where $v_\infty(x)$ satisfies $v''_\infty-v_\infty=0$ in $(0,L)$, therefore $v_\infty=\mathcal B_\infty \cosh (x-L)$ with $\mathcal B_\infty=\frac{\bar u L}{\sinh L}$ follows from the fact that $\int_0^L v_\infty (x)dx=\bar u L$.

\begin{figure}[ht]
        \centering
\includegraphics[width=\textwidth,height=3in]{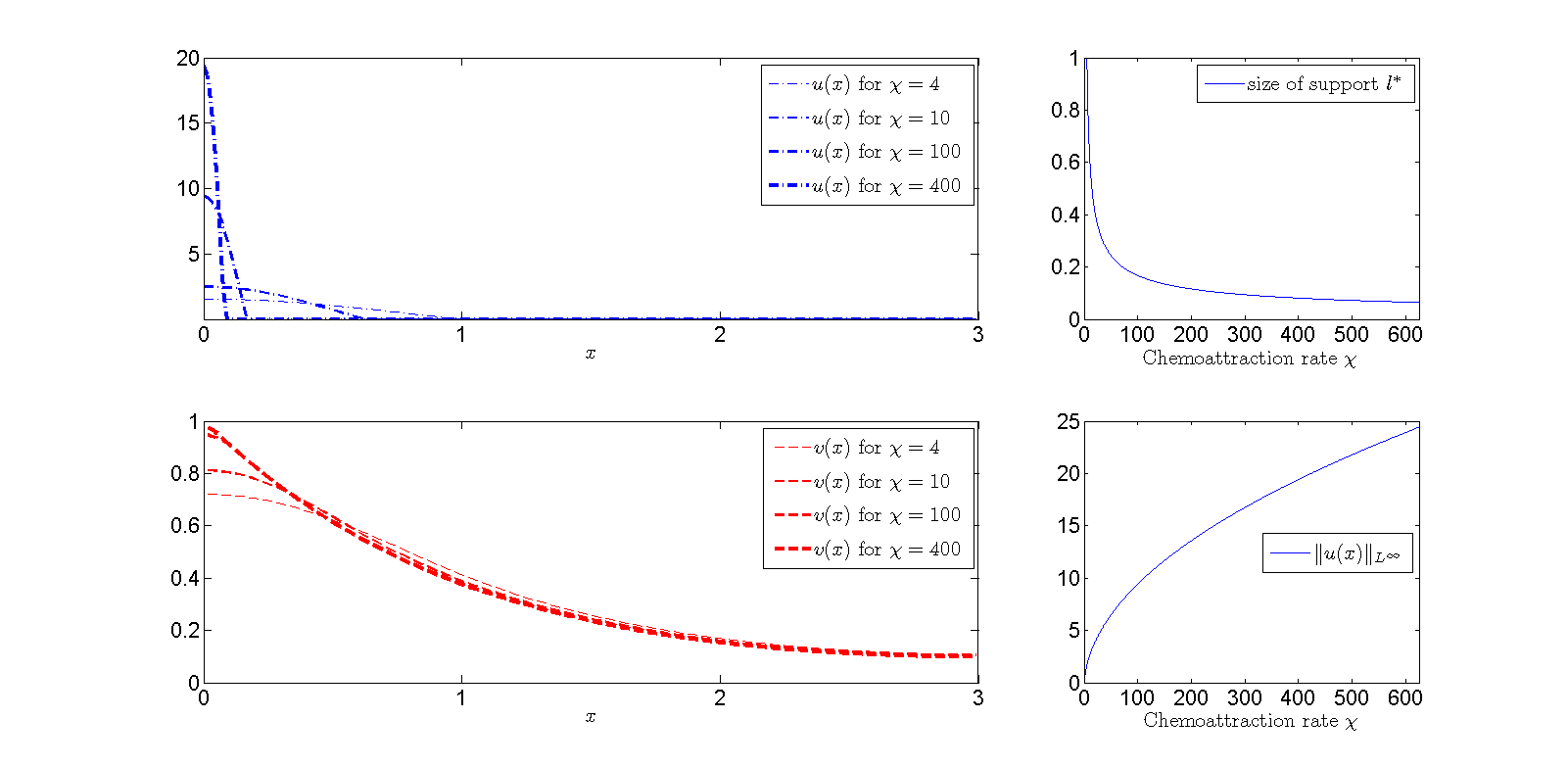}
  \caption{Left panel: Plot of the solution $(u(x),v(x))$ for $\chi=4,10,100$ and $400$.  It is observed that $u(x)\rightarrow \delta(x)$ and $v(x)\rightarrow \frac{\cosh(3-x)}{\sinh 3}$ pointwisely as $\chi\rightarrow \infty$.  Right panel: Plot of the asymptotic convergence of $l^*$ and $\Vert u\Vert_{L^\infty}$ with respect to $\chi$, where we observe that as $\chi\rightarrow \infty$, the support size $l^*$ of $u(x)$ shrinks to zero while the maximum of $u(x)$ grows to $\sqrt{\chi-1}$.  Overall, the chemotaxis rate $\chi$ enhances the formation of (boundary) spikes, namely  the boundary cell aggregation for these solutions.}\label{figure2}
\end{figure}

\begin{remark}
Indeed, it is known that the Green's function $G(x;x_0)$ of
\[
\left\{\begin{array}{ll}
-G''+G=\delta(x;x_0),&x\in(0,L),\\
G'(0;x_0)=G'(L;x_0)=0.&
\end{array}
\right.
\quad
\mbox{is}
\quad
G(x;x_0)=
\left\{\begin{array}{ll}
\frac{\cosh(L-x_0)}{\sinh L}\cosh x,&x\in(0,x_0),\\
\frac{\cosh x_0}{\sinh L}\cosh(L-x),& x\in (x_0,L).
\end{array}
\right.
\]
Now since $u(x)\rightarrow\bar u  \delta_0(x)$, we have that $v(x)\rightarrow \bar u  G(x;0)$, which is the same as what we obtain above.  See Figure \ref{figure2} for illustration.
\end{remark}

\subsection{Similar-bumps}

In the preceding section, we have shown for each $\chi\in(\chi_1,\chi_2)$,  all the non--constant solution $(u,v)$ of (\ref{12}) must be monotone and given by (\ref{28}) and (\ref{212}) or its reflection $(u(x),v(x))$ or $(u(L-x),c(L-x))$.  We now proceed to look for non--monotone solutions of (\ref{12}) and we shall assume $\chi>\chi_2$ from now on.  Note that that when $\chi=\chi_k$, $k\geq2$, there exist solutions of (\ref{12}) explicitly given by the one--parameter family (\ref{21}) such that both $u(x)$ and $v(x)$ are positive in $(0,L)$.  It is our goal to investigate solutions such that $u(x)$ is compactly supported, i.e., it vanishes in subset(s) of $(0,L)$ with a positive measure.  The main results of this section can be summarized as follows: for $\chi\in(\chi_2,\infty)$, there exists non--monotone solutions; and for each $\chi\in(\chi_k,\chi_{k+1})$, $k\geq2$, the sign of $v'(x)$ changes at most $k-1$ times in $(0,L)$.

Hereafter by ``similar-bump solutions'' we mean the bump-profile solutions with the same amplitude whose profile will be generally given later in (\ref{similar-bump}).
For simplicity, we depart with the construction of non--monotone solutions with two similar half-spike profiles, denoted by $(u_2(x),v_2(x))$, and approach as follows: (i) find $(u,v)$ over $(0,\frac{L}{2})$ as in the previous section; (ii) reflect it about $x=\frac{L}{2}$ to obtain solution of (\ref{12}) over the whole interval.  Then $u$ and $v$ are symmetric about $\frac{L}{2}$ with $v'(\frac{L}{2})=0$.

Since $\chi>\chi_2$, by the same arguments as for (\ref{27}), there exists $l_2^*\in(\frac{\pi}{2\omega},\frac{\pi}{\omega})$, which is the size of support of $u(x)$ in $(0,\frac{L}{2})$, such that
$\frac{1}{\omega}\tan \omega l_2^*=\tanh \big(l_2^*-\frac{L}{2}\big).$
Then we solve (\ref{12}) over $(0,\frac{L}{2})$ to find its solution $(u,v)$ given by
\[
\mathbb U_2(x)=\left\{\begin{array}{ll}
\mathcal A_2\big(\cos \omega x-\cos \omega l_2^* \big),&x\in(0,l_2^*),\\[2mm]
0,& x\not \in (0,l_2^*),
\end{array}
\right.
\quad
\mbox{and}
\quad
\mathbb V_2(x)=\left\{\begin{array}{ll}
\mathcal A_2\big(\frac{1}{\chi}\cos \omega x-\cos \omega l_2^* \big),&x\in(0,l_2^*),\\[2mm]
\mathcal B_2\cosh (x-\frac{L}{2}),& x\in (l_2^*,\frac{L}{2}),\\
0,& x\not \in (0,\frac{L}{2}),
\end{array}
\right.
\]
with
\begin{equation*}
\mathcal A_2=\frac{\bar u L/2}{\frac{1}{\omega}\sin \omega l_2^*-l_2^*\cos \omega l_2^*}, \
\mathcal B_2=\frac{\bar u L(\frac{1}{\chi}-1)/2}{\big(\frac{1}{\omega}\tan \omega l_2^*-l_2^*\big)\cosh (l_2^*-\frac{L}{2})},
\end{equation*}

By reflecting $(\mathbb U_2,\mathbb V_2)$ about $x=\frac{L}{2}$, we can find that the solutions of our interest must be one of the pairs $(u_2^\pm,v_2^\pm)$ with
$
(u^+_2,v^+_2)(x)=(\mathbb U_2,\mathbb V_2)(x)+(\mathbb U_2,\mathbb V_2)(L-x)
$
or
$
(u^-_2,v^-_2)(x)=(\mathbb U_2,\mathbb V_2)(L/2-x)+(\mathbb U_2,\mathbb V_2)(x-L/2).
$
Note that $(u_2^+,v_2^+)$ correspond to a double--boundary--spike solution, $(u_2^-,v_2^-)$ the single--interior--spike solution.  See the first column in Figure \ref{figure3} for illustration of $(u_2^\pm,v_2^\pm)$.

Next we extend $(u_2,v_2)$ obtained above to multiple half--bumps $(u_k,v_k)$ with $k\geq3$, for which we assume that $\chi>\chi_k$, $k\geq3$ from now on.  Similar as above, our strategy is to first solve (\ref{12}) over $(0,\frac{L}{k})$ for $(\mathbb U_k,\mathbb V_k)$ and continuously reflect this half-spike bump profile at $x=\frac{L}{k}, \frac{2L}{k},\cdots$ until it eventually extends to the whole interval $(0,L)$, as to be realized mathematically shown below.
Since $\chi>\chi_k$, there exists a unique  $l_k^*\in (\frac{\pi}{2\omega},\frac{\pi}{\omega})$ such that
$\frac{1}{\omega} \tan \omega l_k^*=\tanh \Big(l_k^*-\frac{L}{k}\Big).$
Then we can find that
\[
\mathbb{U}_k(x)=\left\{\begin{array}{ll}
\mathcal A_k\big(\cos \omega x-\cos \omega l_k^* \big),&x\in(0,l_k^*),\\
0,& x\not\in(0,l_k^*),
\end{array}
\right.
\quad
\mbox{and}
\quad
\mathbb{V}_k(x)=\left\{\begin{array}{ll}
\mathcal A_k\big(\frac{1}{\chi}\cos \omega x-\cos \omega l_k^* \big),&x\in(0,l_k^*),\\
\mathcal B_k\cosh (x-\frac{L}{k}),& x\in (l_k^*,\frac{L}{k}),\\
0, & x\not \in (0,\frac{L}{k}).
\end{array}
\right.
\]
where
\begin{equation*}
\mathcal A_k=\frac{\bar u L/k}{\frac{1}{\omega}\sin \omega l_k^*-l_k^*\cos \omega l_k^*},
\mathcal B_k=\frac{\bar u L(\frac{1}{\chi}-1)/k}{\big(\frac{1}{\omega}\tan \omega l_k^*-l_k^*\big)\cosh (l_k^*-\frac{L}{k})},
\end{equation*}
Then $(\mathbb U_k,\mathbb V_k)$ solves (\ref{12}) over $(0,\frac{L}{k})$ with $\int_0^\frac{L}{k}\mathbb U_k(x)dx=\frac{1}{k}$.  By reflecting and extending it at $x=\frac{2L}{k},\frac{3L}{k}...$, we find that the following two pairs $(u_k^{\pm},v_k^{\pm})$  solve (\ref{12}) over $(0,L)$:
\begin{eqnarray}\label{similar-bump}
\begin{aligned}
&(u^+_k, v^+_k)(x)=\sum_{i=0}^{[\frac{k}{2}]}  (\mathbb{U}_k,\mathbb{V}_k)\Big(\frac{2iL}{k}-x\Big)+(\mathbb{U}_k, \mathbb{V}_k)\Big(x-\frac{2iL}{k}\Big),\\
&(u^-_k,v^-_k)(x)=\sum_{i=1}^{[\frac{k}{2}]+1}  (\mathbb{U}_k,\mathbb{V}_k)\Big(\frac{(2i-1)L}{k}-x\Big)+(\mathbb{U}_k,\mathbb{V}_k)\Big(x-\frac{(2i-1)L}{k}\Big),
\end{aligned}
\end{eqnarray}
It is easy to check that both $u_k^{\pm}$ and $v_k^{\pm}$ have $k$ half--bumps.  See Figure \ref{figure3} for illustrations of these similar profiled solutions.

\begin{figure}[ht]
        \centering
\includegraphics[width=\textwidth,height=3in]{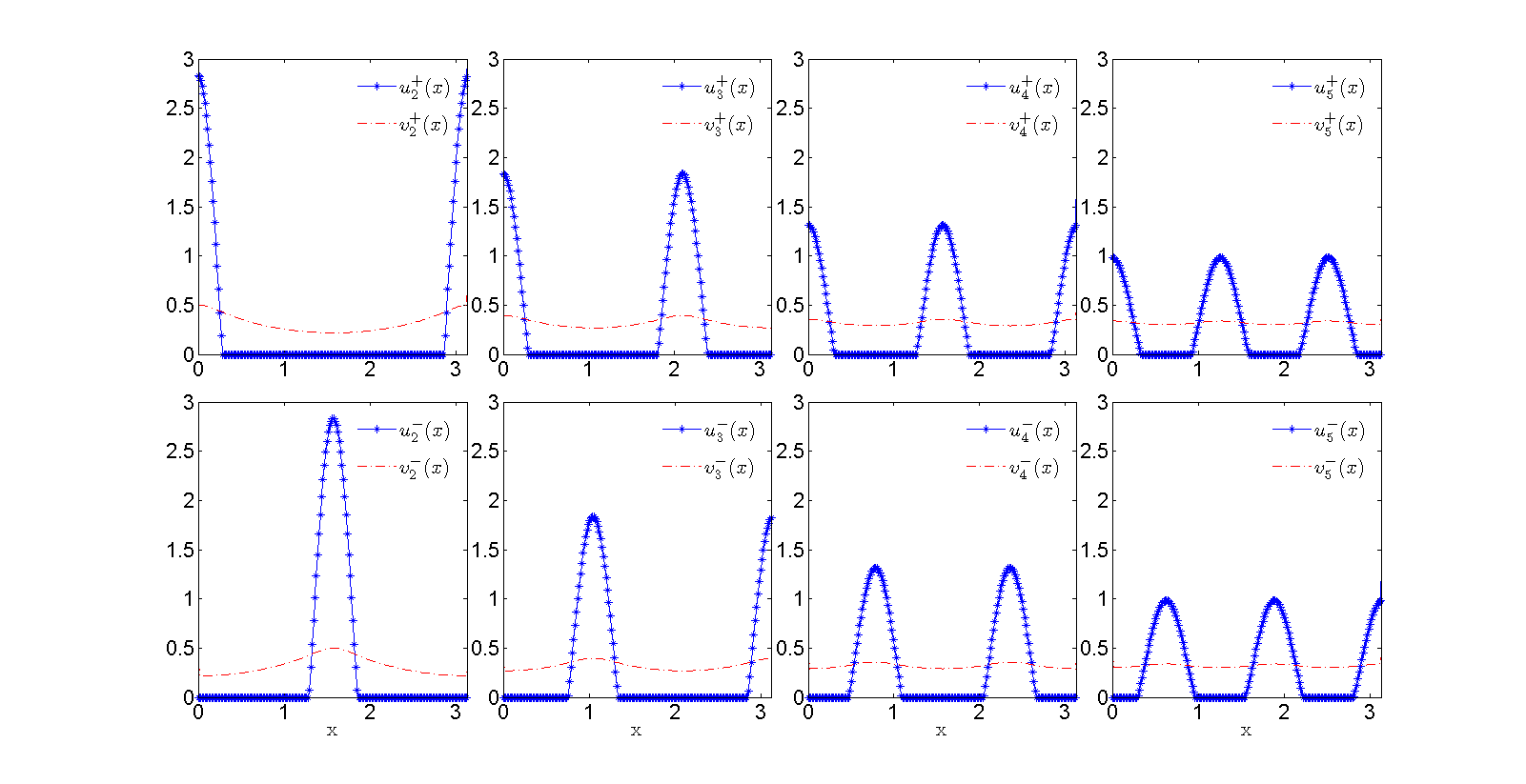}
  \caption{Plot of the pairs $(u_k^\pm,v_k^\pm)$ over $(0,\pi)$ for $k=2,3,4,$ and $5$ in columns from the left to the right, $k$ corresponding to the number of half--spikes, where $\chi=40$ and a unit total cell population are chosen.}\label{figure3}
\end{figure}

In summary, our similar--bump solutions can be summarized as follows.

\begin{proposition}
Let $n_0\geq 1$ be an arbitrary positive integer, then (\ref{12}) has a unique pair of similar--bump solutions $(u_k^\pm,v_k^\pm)(x)$ with $n_0$ half--bumps, explicitly given by (\ref{similar-bump}), if and only if $\chi\geq\chi_{n_0}$. Moreover, if $\chi\geq \chi_{n_0}$ and $n_0\geq 2$, (\ref{12}) there also exist similar--bump solutions for each $k=1,...,n_0-1$.
\end{proposition}

It seems necessary to point out that, similar as above, one can show that $\Vert u_k\Vert_{L^\infty}$ is strictly increasing in $\chi$, $\Vert u_k\Vert_{L^\infty}=\frac{\omega}{\sqrt{k}}+o(1)$ as $\chi\rightarrow \infty$, and $u_k$ converges to a linear combination of Dirac delta function.  Moreover, one can find that $\Vert u_1\Vert_{L^\infty}>\Vert u_2\Vert_{L^\infty}>...>\Vert u_k\Vert_{L^\infty}$. A numerical illustration of these facts is presented in Figure \ref{figure3}. Let us finally mention that again these $k$ half-bumps can be considered intuitively as bifurcation curves emanating from the corresponding $k$ half-bumps touching zero at the end points of the interval $[0,L]$ existing for the first time at $\chi=\chi_k$.

\subsection{Asymmetric multi--bump and interior--bump solutions}
We point out that for each $\chi>\chi_2$, (\ref{12}) also admits asymmetric multi--bump solutions, which we shall denote by $(u^\#_2,v^\#_2)$, as it has been done in Section 4.5 of \cite{BCR} for two half--bumps.  Though not done or stated there explicitly, one can easily see that their work covers the asymmetric bumps with more than 3 aggregates-see Proposition 4.7 in \cite{BCR}.

Let us start by finding the explicit formula of $(u^\#_2,v^\#_2)$ as follows: for each $\chi>\chi_2$, one can chooses an arbitrary $L_0\in(\frac{\pi}{\omega},L-\frac{\pi}{\omega})$ and find $l^*$ and $l^{**}$ such that $\frac{1}{\omega}\tan \omega l^*=\tanh (l^*-L_0)$ and $\frac{1}{\omega}\tan \omega l^{**}=\tanh \big(l^{**}-(L-L_0)\big)$.
Then we the two aggregates supported by $(0,l^*)$ and $(L-l^{**},L)$ and find that
\[
u^\#_2(x)=\left\{\begin{array}{ll}
\mathcal A_l\big(\cos \omega x-\cos \omega l^* \big),&x\in(0,l^*),\\
0,& x\in (l^*,L-l^{**}),\\
\mathcal A_r\big(\cos \omega (x-L)-\cos \omega l^* \big),&x\in(L-l^{**},L),\\
\end{array}
\right.
\]
and
\[
v^\#_2(x)=\left\{\begin{array}{ll}
\mathcal A_l\big(\frac{1}{\chi}\cos \omega x-\cos \omega l^* \big),&x\in(0,l^*),\\
\mathcal B_l\cosh (x-L_0),& x\in (l^*,L_0),\\
\mathcal B_r\cosh (x-L_0),& x\in (L_0,L-l^{**}),\\
\mathcal A_r\big(\frac{1}{\chi}\cos \omega (x-L)-\cos \omega l^{**} \big),&x\in(L-l^{**},L),\\
\end{array}
\right.
\]
with $m_1$ and $m_2$ being the cell population on the left and right aggregates respectively
\begin{equation*}
\mathcal A_l=\frac{m_1}{\frac{1}{\omega}\sin \omega l^*-l^*\cos \omega l^*}, \mathcal A_r=\frac{m_2}{\frac{1}{\omega}\sin \omega l^{**}-l^{**}\cos \omega l^{**}},
\end{equation*}
and
\begin{equation*}
\mathcal B_l=\frac{m_1(1-\frac{1}{\chi})}{(l^*-\frac{1}{\omega}\tan \omega l^*)\cosh(L_0-l^*)}, \mathcal B_r=\frac{m_2(1-\frac{1}{\chi})}{(l^{**}-\frac{1}{\omega}\tan \omega l^{**})\cosh(L-L_0-l^{**})};
\end{equation*}
moreover, the continuity of $v(x)$ at $L_0$ implies that $\mathcal B_l=\mathcal B_r$, i.e.,
\[\frac{m_1}{(l^*-\frac{1}{\omega}\tan \omega l^*)\cosh(L_0-l^*)}=\frac{m_2}{(l^{**}-\frac{1}{\omega}\tan \omega l^{**})\cosh (L-L_0-l^{**})}.\]

In terms of the asymmetric multi--bump solutions, we are able to construct solutions with more aggregates and complex patterns.  See Figure \ref{figure5} for illustration.
\begin{figure}[ht]
        \centering
\includegraphics[width=\textwidth,height=2.4in]{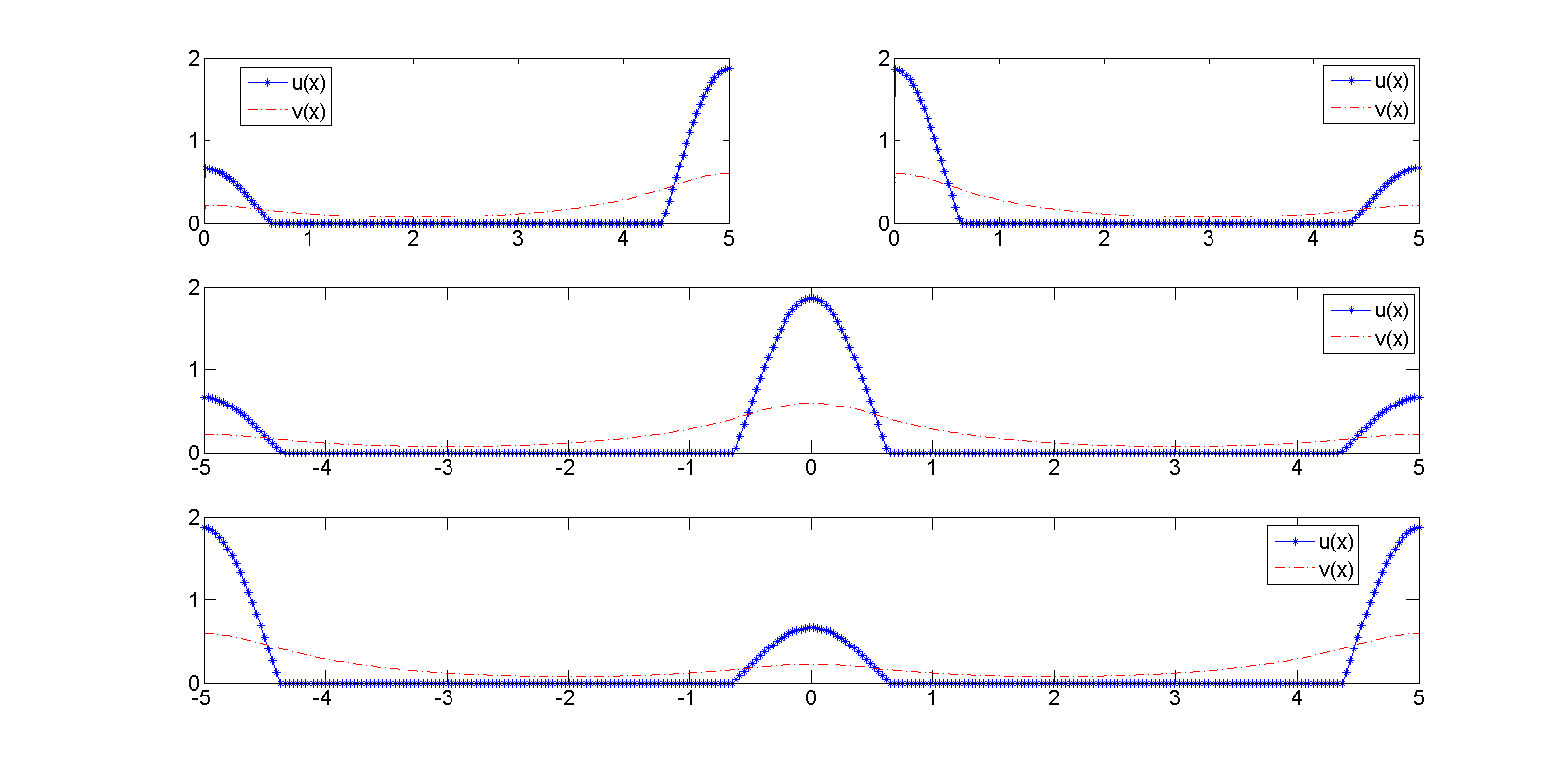}
  \caption{In the top row, we plot $u^\#_2(x)$ with asymmetric double boundary spikes and its reflection $u^\#_2(L-x)$ over $(0,5)$ with $\chi=10$.  Here we choose $L_0=3$ and find that $l^*=0.6326$ and $l^{**}=0.6449$.  In the middle and bottom rows, we piece together the two aggregates at the large and small spikes to obtain solutions with large and small interior spikes.}\label{figure5}
\end{figure}

In contrast to the similar--bump case, in which one can combine the two boundary half--bumps to a single interior bump,  Proposition 4.8 in \cite{BCR} states that if $u(x)$ is a single interior spike in $(0,L)$, then it is symmetric about $x=\frac{L}{2} \triangleq L_0$.  One can show that this holds true for multi--half bumps.  We summarize the results as follows.

\begin{proposition}
Let $n_0\geq 2$ be an arbitrary positive integer.  If $\chi>\chi_{n_0}$, (\ref{12}) has infinitely many asymmetric multi--bump solutions $(u_k^\#,v_k^\#)$ that have $k$ half--bumps, for each $k=2,...,n_0-1$, and the graph of $u^\#_k$ is symmetric within each connected component of its support in $(0,L)$, except the half--bumps on the boundaries; moreover, if $\chi<\chi_{n_0+1}$, then any solution $(u,v)$ of (\ref{12}) has at most $n_0$ half--bumps.
\end{proposition}

As can be easily computed there always exist multi--bump solutions and the spatial--temporal dynamics are rich and complex as shall be seen later in numerical simulations when $\chi$ is large.


\section{Gradient-Flow Structure}
It is known that the system (\ref{11}) has the following free energy
\begin{equation}\label{41}
\mathcal E(u(x,t),v(x,t))=\frac{1}{\chi}\int u^2dx+\int (v_x^2+v^2-2 uv) \,dx,
\end{equation}
that is non--increasing along the solution trajectory of (\ref{11}), and its dissipation is given by
\begin{equation}\label{42}
  \frac{d\mathcal E}{dt}=-\frac{2}{\chi}\int  u |( u -\chi v)_x|^2dx- 2\int |v_t|^2dx:= - \mathcal I\leq 0, \quad \mbox{for all } t>0.
\end{equation}
The free energy $\mathcal E$ allows for a gradient flow structure of this system in a product space, see \cite{CC,CLM,BL,BCKKLL}. The hybrid gradient flow structure introduced in \cite{BL,BCKKLL} treats the evolution of the cell density in probability measures while the evolution of the chemoattractant is done in the $L^2$-setting. The gradient flow structure used in probability densities follows the blueprint of the general gradient flow equations treated in \cite{AGS,Vil_1,CMcCV03}. Moreover, solutions were proved to be unique among the class of bounded densities \cite{CLM}.

Note that the free energy $\mathcal E$ is a Liapunov functional since steady states $(u_s, v_s)$ are characterized by zero dissipation $\mathcal I(u_s, v_s)=0$. Therefore, we readily obtain that for any steady state $(u_s, v_s)$, the quantity $u_s -\chi v_s$ must be constant in each connected component of the support of the cell density $u_s$.  Note that the $v$--equation of (\ref{12}) readily gives us
\[
\int_0^L \big[(v_s)_x^2+v_s^2\big]dx=\int_0^L u_sv_s\, dx\,,
\quad
\mbox{which implies that}
\quad
\mathcal E(u_s, v_s)=\frac{1}{\chi}\int_0^L u_s (u_s-\chi v_s)\, dx.\]
Moreover, since $u_s-\chi v_s=\lambda _i$ for some constants $\lambda _i$ on each of
the possibly countable many connected components of the support of $u_s$, denoted by $\text{sppt}_i$, we have that
\begin{equation}\label{44}
\mathcal E(u_s,v_s)=\frac{1}{\chi}\sum_{i}\int_{\text{sppt}_i} \lambda _i u_s\, dx.
\end{equation}
Notice that all the constructed stationary states in the previous section have finitely many connected components in their support and the total number of connected components is less than $n_0$, if $\chi<\chi_{n_0}$, as discussed above. We now study the energy of the steady states constructed above.  First of all, we find that the constant solution $(\bar u,\bar v)$ has free energy
$
\mathcal E(\bar u, \bar v)=\frac{(1-\chi)\bar u}{\chi}M=-\frac{\omega^2M^2}{\chi L},
$
and the boundary spike $(u_1,v_1)$ has free energy
\[\mathcal E(u_1, v_1)=\frac{\cos \omega l^*}{\frac{1}{\omega}\sin  \omega l^*-l^*\cos  \omega l^*}\frac{\omega^2M^2}{\chi}=\frac{\omega}{\tan z-z}\frac{\omega^2M^2}{\chi},\]
where $z:=\omega l^*$.
We claim that $\mathcal E(u_1, v_1)<\mathcal E(\bar u, \bar v)$.  To show this, we first note that $-l^*\tan z<(L-l^*)z$.  Indeed, $\tan z=\omega \tanh(l^*-L)$ thanks to (\ref{27}), then the fact $\tanh (L-l^*)<L-l^*$ readily implies that $-z\tanh(l^*-L)<(L-l^*)z$, which leads to this inequality.  Then we can find
\[
\mathcal E(u_1, v_1)-\mathcal E(\bar u, \bar v)=\left(\frac{l^*}{L}-\frac{z}{z-\tan z}\right)\frac{\omega^2 M^2}{\chi l^*}<0,
\]
which is our claim.

\subsection{Symmetric multi--bump solution}
According to \cite{BCR}, one can compare the free energies of half--bump $(u^1,v^1)$, two similar half--bumps $(u^2,v^2)$ and four similar half--bumps $(u^4,v^4)$ and show that $\mathcal E(u^1,v^1)<\mathcal E(u^2,v^2)<\mathcal E(u^4,v^4)<\mathcal E(\bar u,\bar v)$. In this section, we provide a complete hierarchy of the free energies of all similar multi--bumps.

According to (\ref{44}), for similar bump solutions $(u_k(x),v_k(x))=: (u_k^{\pm}(x),v_k^{\pm}(x))$ given in \eqref{similar-bump} the associated free energy is
\begin{equation*}
\mathcal E(u_k(x),v_k(x)) =\frac{1}{k(\tan \omega l_k^*-\omega l_k^*)}\frac{\omega^3 M^2}{\chi}.
\end{equation*}
Our next result shows that boundary spike has the least energy among all similar bump solutions.
\begin{lemma}\label{lemma41}
Assume that $\chi>\chi_k$, $k\geq1$, and let $(u_k(x),v_k(x))$ be the symmetric multi--bump solution.  Then we have the following inequalities
\begin{equation}\label{46}
\mathcal E(u_1,v_1)<\mathcal E(u_2,v_2)<...<\mathcal E(u_k,v_k)<\mathcal E(\bar u,\bar v).
\end{equation}
\end{lemma}

\begin{figure}[ht]
        \centering
\includegraphics[width=\textwidth,height=2.5in]{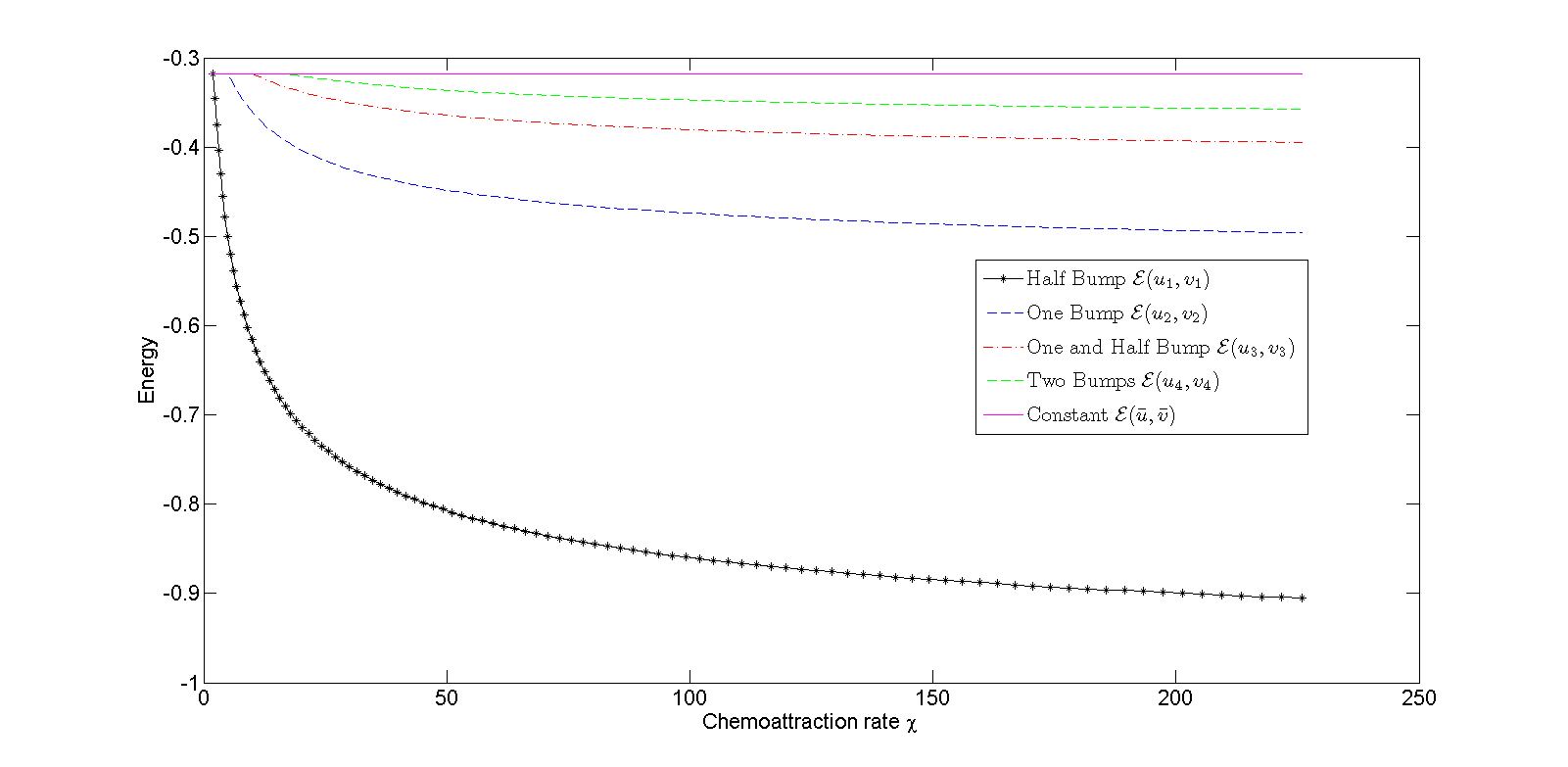}
  \caption{Hierarchy and qualitative behaviors of the steady state free energy $\mathcal E$ in (\ref{42}). Here a unit total cell population and $L=6$ are chosen.  This illustrates Lemma \ref{lemma41} showing that the single boundary spike has the least energy among all steady states with similar profiles, those with more half--bumps have larger energies, and the constant solution $(\bar u,\bar v)$ has the largest energy.  We also observe that for each $k\in\mathbb N^+$, $\mathcal E(u_k,v_k)$ approaches a constant as $\chi\rightarrow \infty$.  This fact is rigorously proved latter.}\label{figure6}
\end{figure}

\begin{proof}
The last inequality can be verified by the same arguments as for $\mathcal E(u_1,v_1)<\mathcal E(\bar u,\bar v)$.  Let us denote
\[F(L):=\frac{L}{\frac{1}{\omega}\tan \omega l^*-l^*}.\]
Then $\mathcal E(u_k,v_k)=\frac{\chi}{\omega^3 M^2} F(\frac{L}{k})$ and (\ref{46}) is equivalent as
$F(L)<F(L/2)<...<F(L/k),$
and it is sufficient to prove that $\frac{\partial F(L)}{\partial L}<0$.  Note that $l^*$ depends on $L$ in $F(L)$ but not on the fixed constant $\omega$.
Denote $z=\frac{1}{\omega} \tan \omega l^*$, $z\in[0,1)$ and rewrite $F$ as
$
F(L)=\frac{l^*-\text{arctanh} z}{z-l^*}.
$
Note that $l^*$ depends on $L$, then differentiating $F(L)$ with respect to $L$ gives us
\[
\frac{\partial F}{\partial L}=\frac{\partial F}{\partial l^*}\frac{\partial l^*}{\partial L},\quad
\mbox{where} \quad \frac{\partial l^*}{\partial L}=-\frac{\cos^2 \omega l^*}{\cosh^2(l^*-L)-\cos^2\omega l^*}<0\,,
\]
then we only need to show $\frac{\partial F}{\partial l^*}>0$ to conclude the lemma.
We calculate, by using the fact that $\frac{\partial z}{\partial l^*}=1+(\omega z)^2$, to find
\begin{align*}
\frac{\partial F}{\partial l^*}
=&\frac{\big(1-\frac{1}{1-z^2}\frac{\partial z}{\partial l^*}\big)(z-l^*)-(l^*-\text{arctanh} z)(\frac{\partial z}{\partial l^*}-1)  }{(z-l^*)^2} = \frac{(1+\omega^2)z^2}{(1-z^2)(z-l^*)^2}\psi (l^*), \quad l^*\in(\frac{\pi}{2\omega},\frac{\pi}{\omega}),
\end{align*}
where
\begin{equation*}
\psi(l^*):=l^*-z-\frac{\omega ^2}{1+\omega ^2} (l^*-\text{arctanh} z)(1-z^2).
\end{equation*}
Now we prove that $\psi(l^*)>0$.  Note that $\psi(\frac{\pi}{\omega})=\frac{\pi}{\omega}>0$, then it is sufficient to show that $\psi'(l^*)<0$.  Indeed, we have for $l^*\in(\frac{\pi}{2\omega},\frac{\pi}{\omega})$ that
\begin{align*}
\psi'(l^*)
=&1-\frac{\partial z}{\partial l^*}-\frac{\omega^2}{1+\omega^2}\Big(\big(1-\frac{1}{1-z^2}\frac{\partial z}{\partial l^*}\big)(1-z^2)-2z(l^*-\text{arctanh} z)\frac{\partial z}{\partial  l^*}  \Big) \nonumber \\
=&\frac{2\omega^2 z(1+(\omega z)^2)}{1+\omega^2} (l^*-\text{arctanh} z).
\end{align*}
We claim that $\phi(l^*):=l^*-\text{arctanh} z>0$ for $l^*\in(\frac{\pi}{2\omega},\frac{\pi}{\omega})$; indeed $\phi(\frac{\pi}{\omega})=\frac{\pi}{\omega}>0$ and \[\phi'(l^*)=1-\frac{1}{1-z^2}\frac{\partial z}{\partial l^*}=-\frac{(1+\omega^2)z^2}{1-z^2}<0,\]
therefore $\phi(l^*)>0$ as claimed and $\psi'(l^*)>0$.  This verifies that $\frac{\partial F}{\partial l^*}>0$ hence the lemma is proved.
\end{proof}

An immediate consequence is the following result.

\begin{lemma}
Assume that $\omega>\frac{k\pi}{L}$.  Let $(u_k(x),v_k(x))$ be the $k$th-symmetric multi--bump solution.  Then we have the following inequalities
\begin{equation*}
\Vert u_1(x)\Vert_{L^\infty}>\Vert u_2(x)\Vert_{L^\infty}>...>\Vert u_k(x)\Vert_{L^\infty}.
\end{equation*}
\end{lemma}
\begin{proof}
We already know that
\[\Vert u_k(x)\Vert_{L^\infty}=\frac{\frac{1}{\cos w l^*_k}-1}{k(\frac{1}{\omega}\tan \omega l_k^*-l_k^*)}.\]
Let $M(L)=\frac{(\frac{1}{\cos w l^*_k}-1)L}{\frac{1}{\omega}\tan \omega l_k^*-l_k^*}$, then it is sufficient to prove that $\frac{\partial M}{\partial L}>0$ for the lemma.  To this end, we write $g(L)=\frac{1}{\cos w l^*_k}-1$, then $M(L)=g(L)F(L)$ and
$M'(L)=g'(L)F(L)+g(L)F'(L)$;
note that $g'(L)=\frac{\omega \sin \omega l_k^*}{\cos^2 \omega l_k^*}\frac{\partial l_k^*}{\partial L}<0$, $F(L)<0$, $g(L)<0$ and $F'(L)<0$, then we have that $M'(L)>0$ completing the proof.
\end{proof}

\begin{remark}
Observe that for each $(u_k,v_k)$, $k\in\mathbb N^+$, we know that $l_k^*\rightarrow 0^+$ and $\frac{1}{\omega}\tan \omega l^*=\tanh(l^*-L/k)\rightarrow -\tanh(L/k)$, therefore we can easily find that
\[\mathcal E(u_k,v_k)\rightarrow -\frac{1}{k\tanh(L/k)}, \qquad \mbox{ as }\chi\rightarrow \infty.\]
The qualitative behavior is schematically illustrated in Figure \ref{figure6}.
\end{remark}

\subsection{Asymmetric steady states: the role of $L_0$}
In general, we have that the energy of the asymmetric solution $(u_k^\#,v_k^\#)$ is
$\mathcal E=\frac{1}{\chi}\sum_{i} \lambda _i m_i,$
where $m_i$ are the cell populations on the $i$-th component of the support of $u$.  Let us consider for simplicity the case when there are two asymmetric boundary spikes $(u_2^\#,v_2^\#)$  found in the previous section.

\begin{figure}[ht]
        \centering
\includegraphics[width=\textwidth,height=3.2in]{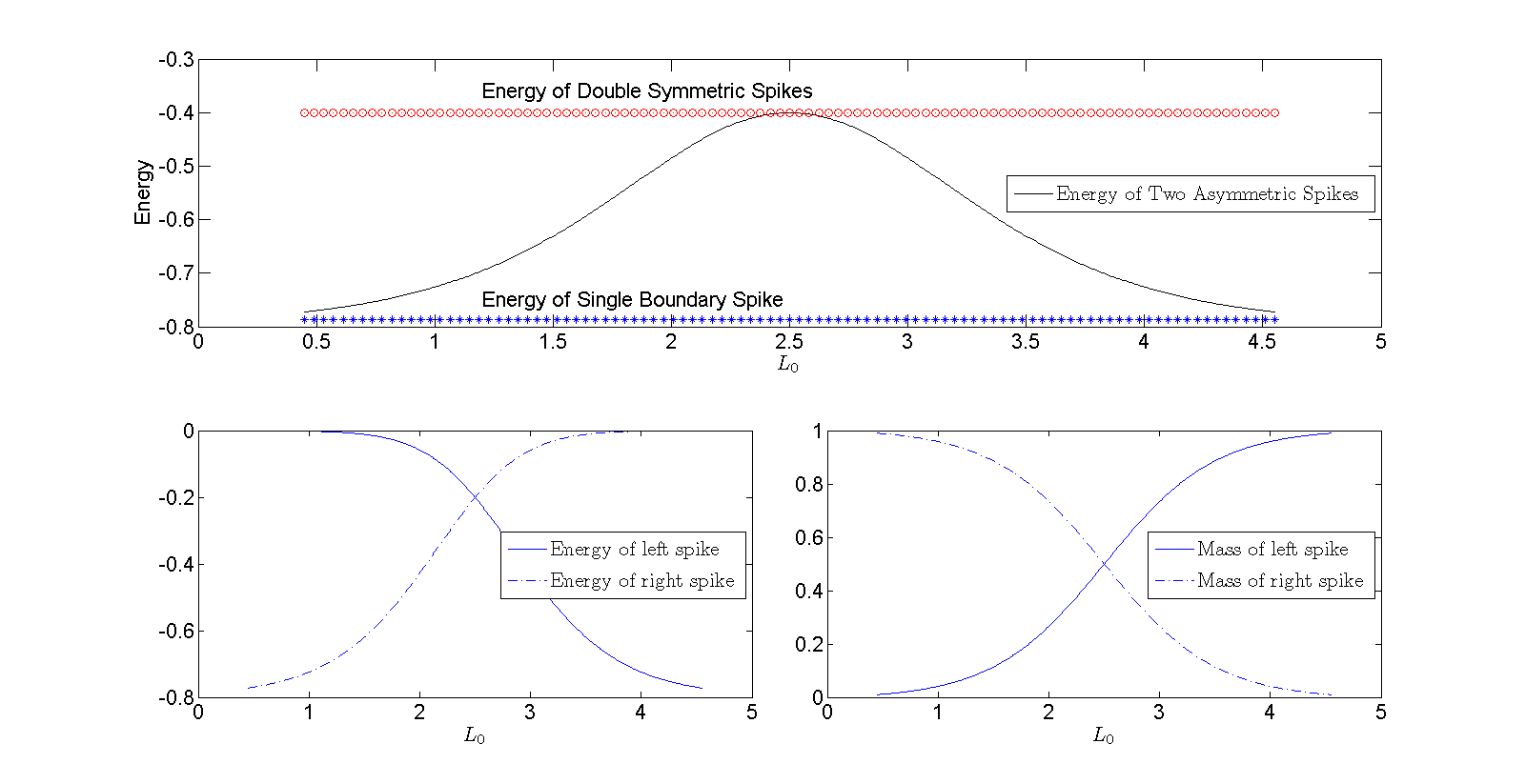}
  \caption{Qualitative behaviors of properties associated with the asymmetric double--boundary--spikes $(u^\#_2,v^\#_2)$ established in Section 3.3.  Here we choose $\chi=50$ and $L=6$.  According to our calculations, for each $L_0\in(0.4488,5.5512)$, one can find steady states $(u^\#_2,v^\#_2)$ such that $u^\#_2$ has two asymmetric boundary spikes.  The top figure illustrates the variation of $E^\#(u,v;L_0)$ as $L_0$ varies.  In particular, $E^\#(u,v;L_0)$ is symmetric about $x=\frac{L}{2}$, and achieves its maximum value there, which apparently equals that of the symmetric double--boundary spike; on the other hand, as $L_0\rightarrow 0.4488$ or $5.5512$, $E^\#(u,v;L_0)$ approaches the energy of the single boundary spike.  In the bottom, we plot the energy and mass of each spike as $L_0$ varies.  It is shown that the free energy of the left spike decreases and its mass increases in $L_0$, while the opposite holds for the right boundary spike.}\label{figure7}
\end{figure}

Therefore, for $(u_2^\#,v_2^\#)$ we have that
\[\mathcal E^\#(u_2^\#,v_2^\#;L_0)=\frac{\omega^2}{2}\Big(\frac{m_1^2}{\frac{1}{\omega}\tan \omega l^*-l^*}+\frac{m_2^2}{\frac{1}{\omega}\tan \omega l^{**}-l^{**}}\Big), L_0\in \Big(\frac{\pi}{\omega}, L-\frac{\pi}{\omega}\Big).\]
On other other hand, let us denote
\[\alpha_1=(l^*-\frac{1}{\omega}\tan \omega l^*)\cosh(L_0-l^*), \alpha_2=(l^{**}-\frac{1}{\omega}\tan \omega l^{**})\cosh (L-L_0-l^{**}),\]
then $m_1=\frac{\alpha_1}{\alpha_1+\alpha_2}$ and $m_2=\frac{\alpha_2}{\alpha_1+\alpha_2}$, and then we can also rewrite
\[\mathcal E^\#(u_2^\#,v_2^\#;L_0)=\frac{\alpha_1\cosh(L_0-l^*)+\alpha_2\cosh (L-L_0-l^{**})}{(\alpha_1+\alpha_2)^2}.\]

In Figure \ref{figure7}, we present the qualitative behaviors of $E^\#(u_2^\#,v_2^\#;L_0)$ as $L_0$ varies. The numerics suggest that  $\mathcal E^\#(u_2^\#,v_2^\#;L_0)$ achieves its maximum at $L_0=\frac{L}{2}$, that we did not pursue analytically.

\begin{remark}
Finally, let us comment that we can approximate the Cauchy problem in the whole space by considering the Neumann problem on the centered interval $[-L/2,L/2]$ and sending $L\to \infty$ with $\chi$ fixed, this can be obtained by rescaling from a case in which the length is fixed and $\chi \to \infty$ with the mass $M$ fixed. It is clear that boundary spikes do not survive in the limit as steady states of the Cauchy problem because then the mass scapes to $\infty$. Moreover, it is easy to check that the distance between bumps in the multi bump solutions with more than one bump without boundary spikes diverges as $L\to \infty$. The conclusion is that the only integrable steady states remaining in the limit are the single bump solutions as shown in \cite{CHVY}.
\end{remark}


\section{Structure Preserving Numerical Scheme and Simulations}

We will adapt the one dimensional first-order finite-volume method for general gradient flow equations developed in \cite{BCF12,CCH15} to equation \eqref{11}. For simplicity, we divide the
computational domain into finite-volume cells $C_j=[x_{j-\frac{1}{2}},x_{j+\frac{1}{2}}]$ of a uniform size $\Delta x$ with $x_j=j\Delta x$, $j\in \{0,\cdots, N\}$,
and denote by
\begin{equation*}
 {\bar u}_j(t)=\frac{1}{\Delta x}\int_{C_j}u(x,t)\,dx,
\end{equation*}
the computed cell averages of the solution $u$ at time $t\geq 0$. The semi-discrete first order finite-volume
scheme is obtained by integrating equation \eqref{11} over each cell $C_j$ and is given by the following system of ODEs for $ {\bar u}_j$:
\begin{equation}
\frac{d {\bar u}_j(t)}{dt}=-\frac{F_{j+\frac{1}{2}}(t)-F_{j-\frac{1}{2}}(t)}{\Delta x},
\label{a2}
\end{equation}
where the numerical flux $F_{j+\frac{1}{2}}$ approximate the continuous flux $-u(u-\chi v)_x$ at cell interface $x_{j+\frac{1}{2}}$. The dependence on $t\geq 0$ is omitted for simplicity. We use the upwind numerical fluxes by computing piecewise constant approximations to $u$ in each cell $C_j$, $\widetilde u_j(x)= {\bar u}_j$, $x\in C_j$, and compute the right (``east''), $u_j^{\rm E}$, and left (``west''), $u_j^{\rm W}$, point values at the cell interfaces $x_{j-\frac{1}{2}}$ and $x_{j+\frac{1}{2}}$, respectively as
\begin{equation}
\begin{aligned}
&u_j^{\rm E}=\widetilde u(x_{j+\frac{1}{2}}-0)=  \bar u_j, \qquad
&u_j^{\rm W}=\widetilde u(x_{j-\frac{1}{2}}+0)= \bar u_j.
\end{aligned}
\label{rew}
\end{equation}
Given the piecewise constant reconstruction $\widetilde u_j(x)$ and point values $u_j^{\rm E},\ u_j^{\rm W}$, the upwind fluxes in \eqref{a2} are computed as
\begin{equation}
F_{j+\frac{1}{2}}=\xi_{j+\frac{1}{2}}^+u_j^{\rm E}+\xi_{j+\frac{1}{2}}^-u_{j+1}^{\rm W},
\label{nf}
\end{equation}
where the discrete values $\xi_{j+\frac{1}{2}}$ of the velocities are obtained using the centered-difference approach,
\begin{equation*}
\xi_{j+\frac{1}{2}}=-\frac{( u_{j+1}-\chi v_{j+1})-( u_{j}-\chi v_{j})}{\Delta x},
\end{equation*}
and $\xi_{j+\frac{1}{2}}^\pm=\pm\max(\pm\xi_{j+\frac{1}{2}},0)$. Here, and in the rest we have simplified the notation avoiding the use of $\bar u_j$ by simply writing $u_j$, considering that the mean value is the approximation of the point value at $x_j$ anyhow given by $u_j=\widetilde u_j(x_j)=\bar u_j$.
Concerning the discretization for the equation of the chemoattractant $v$, we use direct second order finite differences to obtain the scheme
\begin{equation}
\frac{dv_j(t)}{dt}=\frac{v_{j+1}+ v_{j-1}-2 v_{j}}{\Delta x^2} -  v_{j}+ u_{j},
\label{a22}
\end{equation}
Equations \eqref{a2} and \eqref{a22} are supplemented with zero flux boundary conditions meaning
$v_{0}=v_{1}$ and $u_{0}=u_{1}$, and $v_{N}=v_{N-1}$ and $u_{N}=u_{N-1}$, implying
that $F_{\frac{1}{2}}=F_{N-\frac{1}{2}}=0$. Finally, the semi-discrete scheme \eqref{a2} is a system of ODEs, which has to be integrated numerically using a stable and accurate ODE solver.

\begin{remark} We have the following remarks.\\

{\bf Positivity Preserving.-} The scheme \eqref{a2}-\eqref{a22} preserves positivity of the computed cell averages $ {u}_j$ under a CFL condition. The proof is based on the forward Euler step of the ODE system \eqref{a2}, but as usual remains equally valid if the forward Euler method were replaced by higher-order ODE solver as soon as their time stepping is a convex combination of forward Euler steps.
More precisely, the computed cell averages $ {u}_j\ge0$, for all $j$, provided that
the following CFL condition
\begin{equation*}
\Delta t\le\frac{\Delta x}{2a},\quad\mbox{where}\quad a=\max\limits_j\left\{\xi_{j+\frac{1}{2}}^+,-\xi_{j+\frac{1}{2}}^-\right\}\,,
\end{equation*}
is satisfied.\\
\indent {\bf 2nd order Accuracy.-} The method can be turned into second order in a classical way by using slope limiters as in \cite{CCH15} to increase the approximation in \eqref{rew} to second order. In fact, the reconstructed $u$ is given by piecewise linear functions instead of piecewise constant functions and the fluxes are approximated similarly as in \eqref{nf}.
\end{remark}

\subsection{Semidiscrete free energy decay}

A discrete version of the entropy $\mathcal E$ defined in \eqref{41} is given by
\begin{equation}
\mathcal E_\Delta(t)=\Delta x\sum\limits_j\left[\frac1{2\chi} u_j^2 + \frac12 \left( \frac{v_{j+1}-v_j}{\Delta x}\right)^2 -2u_j v_j +\frac12 v_j^ 2\right].
\label{dfe}
\end{equation}
We also introduce the discrete version of the entropy dissipation
\begin{equation}
\mathcal I_\Delta(t)=\Delta x\sum\limits_j \left[ \xi_{j+\frac{1}{2}}^2\min\limits_j(u_j^{\rm E},u_{j+1}^{\rm W}) + \left( \frac{dv_j}{dt}\right)^ 2\right].
\label{ded}
\end{equation}
In the following theorem, we prove that the time derivative of $\mathcal E_\Delta(t)$ is less or equal than the negative of $\mathcal I_\Delta(t)$, mimicking the corresponding property of the continuous relation.

\begin{theorem}\label{dd1}
Consider the system \eqref{11} with no flux boundary conditions on $[0,L]$ with $L>0$ and with initial data $u_0(x)\ge0$. Given the semi-discrete finite-volume scheme \eqref{a2}-\eqref{a22} with $\Delta x=L/N$,
with a positivity preserving piecewise linear reconstruction for $u$ and discrete boundary conditions
$F_{\frac{1}{2}}=F_{N-\frac{1}{2}}=0$. Then, the discrete free energy \eqref{dfe} satisfies
\begin{equation*}
\frac{d}{dt}\mathcal E_\Delta(t)\le-\mathcal I_\Delta(t),\quad \mbox{for all } t>0.
\end{equation*}
\end{theorem}
\begin{proof}
By using \eqref{a2}-\eqref{a22} and discrete integration by parts taking into account the noflux boundary conditions, we get
$$
\frac{d}{dt}\sum\limits_j \frac{\left( v_{j+1}-v_j \right)^2}{2\Delta x}  = - \sum\limits_j \frac{dv_j}{dt} \left(\frac{v_{j+1}-v_j}{\Delta x} - \frac{v_j-v_{j-1}}{\Delta x} \right)= - \Delta x \sum\limits_j \frac{dv_j}{dt} \left( \frac{dv_j}{dt}+v_j-u_j \right)  \,,
$$
$$
 \Delta x \frac{d}{dt} \sum\limits_j \frac{u_j^2}{2\chi} = -\frac1\chi \sum\limits_j F_{j+\frac{1}{2}} (u_j-u_{j+1})\,,
$$
and
$$
\Delta x \frac{d}{dt} \sum\limits_j u_j v_j= -\sum\limits_j F_{j+\frac{1}{2}} (v_j-v_{j+1})+\Delta x \sum\limits_j u_j \frac{dv_j}{dt}\,.
$$
Putting together these identities and collecting terms, we deduce
\begin{align*}
\frac{d}{dt}\mathcal E_\Delta(t)&= -\frac1\chi \sum\limits_j F_{j+\frac{1}{2}} (u_j+\chi v_j-u_{j+1}-\chi v_{j+1}) - \Delta x \sum\limits_j \left( \frac{dv_j}{dt} \right)^2 \\
&= -\frac1\chi \Delta x\sum\limits_j \xi_{j+\frac{1}{2}}F_{j+\frac{1}{2}} - \Delta x \sum\limits_j \left( \frac{dv_j}{dt} \right)^2 .
\end{align*}
Finally, using the definition of the upwind fluxes \eqref{nf} and formula \eqref{ded}, we conclude
\begin{align*}
\frac{d}{dt}\mathcal E_\Delta(t)&=-\Delta x\sum\limits_j \xi_{j+\frac{1}{2}}
\left[\xi_{j+\frac{1}{2}}^+ u_j^{\rm E}+\xi_{j+\frac{1}{2}}^- u_{j+1}^{\rm W}\right]- \Delta x \sum\limits_j \left( \frac{dv_j}{dt} \right)^2\\
&\le-\Delta x\sum\limits_j \xi_{j+\frac{1}{2}}^2\min\limits_j(u_j^{\rm E},u_{j+1}^{\rm W})- \Delta x \sum\limits_j \left( \frac{dv_j}{dt} \right)^2=-\mathcal I_\Delta(t).
\end{align*}\vspace{-0.25cm}
\end{proof}

Let us point out that the decrease of the free energy is crucial to keep at the discrete level the set of stationary states and their stability properties. Due to the decay of the free energy, our semidiscrete scheme is well-balanced since discrete stationary states remain steady and characterized by
$\xi_{j+\frac{1}{2}}=0$ whenever $u_j>0$ and $u_{j+1}>0$.

In the next two subsections, we take advantage of the numerical scheme in order to analyze numerically some interesting phenomena of this problem due to the rich bifurcation structure of the steady states. More precisely, we show the richness on the dynamics of the problem and the subtle choice of the asymptotic state depending on symmetries of the initial data for instance. On the other hand, we show that a metastability behavior appears naturally in the asymptotic behavior for $\chi$ large as the merging or separation of different bumps initially in the solution depends in a very subtle way on the initial value of the chemoattractant leading to a typical staircase behavior in the decay of the free energy in time.

\begin{figure}[ht!]
        \centering
\includegraphics[width=\textwidth,height=3in]{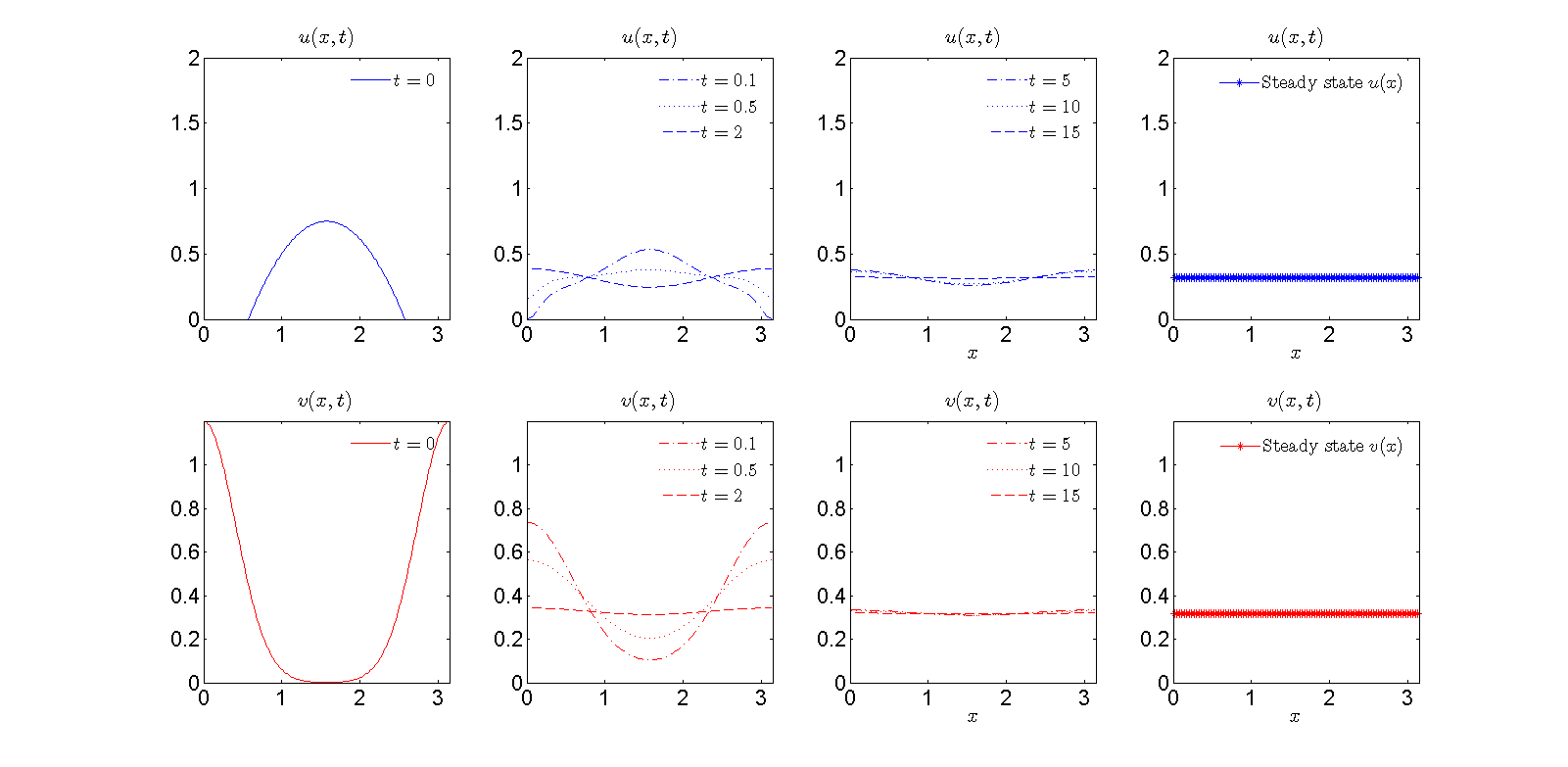}
  \caption{Convergence of symmetric initial data to the constant solution $(\bar u,\bar v)=(\frac{1}{\pi},\frac{1}{\pi})$ for $\chi=4$ and $L=\pi$ with $\chi\in(\chi_1,\chi_2)=(2,5)$. We choose symmetric initial data: $u_0(x)=\max\{0,\frac{3}{4}\big(1-(x-\pi/2)^2\big)\}$ and $v_0(x)=1.2e^{-3x^2}+1.2e^{-3(x-\pi)^2}$.}\label{figure9}
\end{figure}

\subsection{A/Symmetries of the initial data choose the asymptotic behavior}

One can see that as $\chi>0$ increases, the structure of steady states becomes increasingly complex.  Hence we suspect the dynamical behavior of solutions to the system (\ref{11}) shall also become increasingly intricate.  In this section, we use numerical examples to illustrate that the dynamical system (\ref{11}) will exhibit rich behaviors which critically depend on the value of chemotactic sensitivity parameter $\chi$ and/or initial data. When $0<\chi\leq \chi_1$, the dynamics of solutions to (\ref{11}) have been well understood in section 2 and numerical simulations are shown therein.

Next we increase the value of $\chi$ such that $\chi_1<\chi<\chi_2$ to understand how the asymptotic behavior of solutions is chosen dynamically. We fix the initial data for the cell density as $u_0(x)=\max\{0,\frac{3}{4}\big(1-(x-\pi/2)^2\big)\}$ being symmetric on the interval with $L=\pi$. We also choose $\chi=4$ that lies in the interval $(\chi_1,\chi_2)=(2,5)$. We know that for $\chi \in(\chi_1,\chi_2)$ steady states of (\ref{11}) must be either constant or monotone (half-bump), while the constant solution $(\bar u,\bar v)=(\frac{1}{\pi},\frac{1}{\pi})$ is unstable in this case (see details in section 2). On the other hand, for initial data $(u_0,v_0)$ symmetric about $x=\frac{L}{2}$, $(u(x,t),v(x,t))$ stay symmetric for all $t>0$, and the simulation in Figure \ref{figure9} shows that for a symmetric initial data of the chemoattractant taking $v_0(x)=1.2e^{-3x^2}+1.2e^{-3(x-\pi)^2}$, the constant stationary state is the only possible asymptotic limit, and the dynamics illustrate that all symmetric initial data will converge to the constant steady state for $\chi \in(\chi_1,\chi_2)$.  In general, it is natural to expect that the constant solution $(\bar u,\bar v)$ is the global attractor of such symmetric initial data $\chi \in(\chi_1,\chi_2)$, however it is an open problem that deserves future exploration to show that the stable manifold of the constant steady state in this range of values of $\chi$ is given by the symmetric initial data.

\begin{figure}[ht]
        \centering
\includegraphics[width=\textwidth,height=3in]{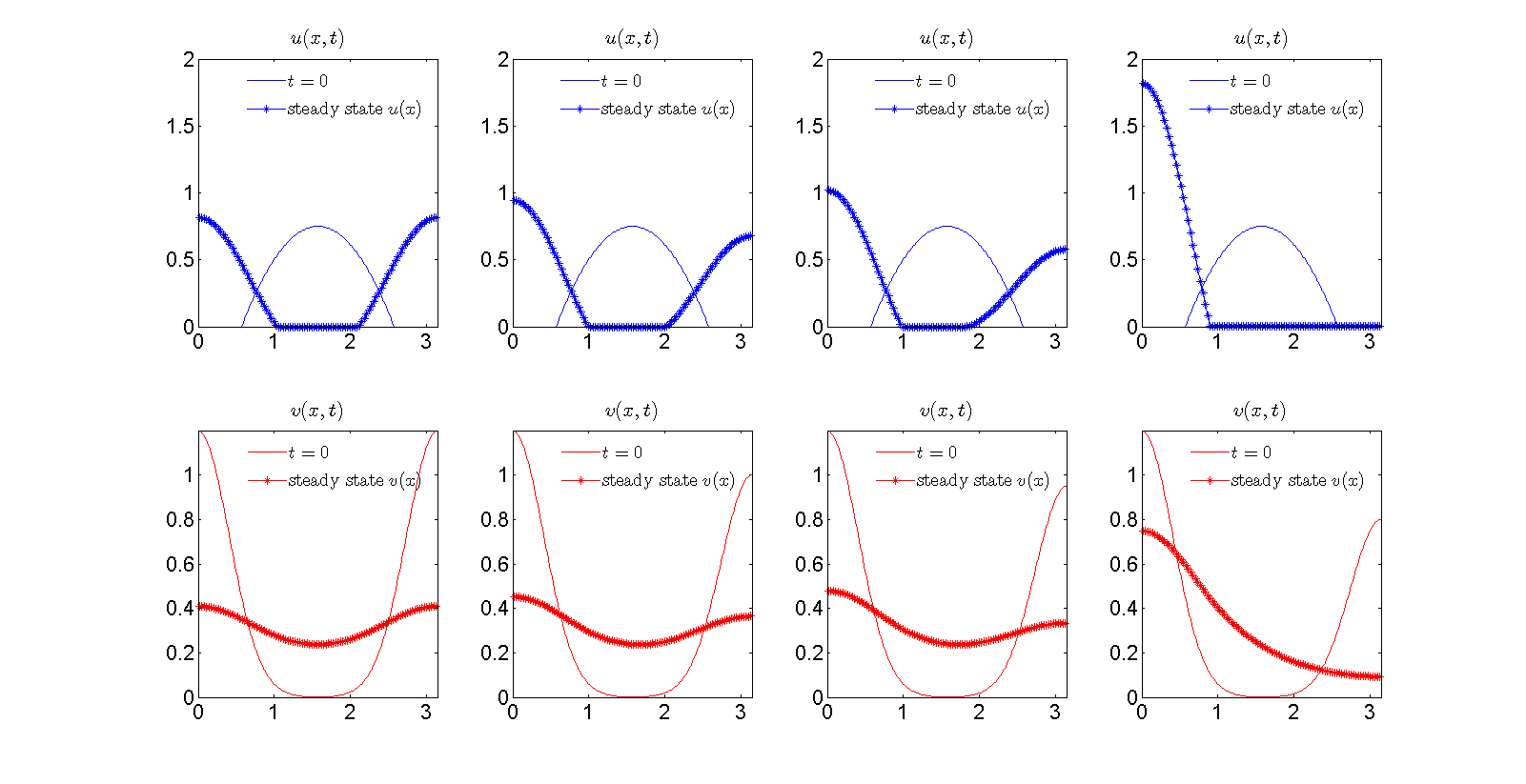}
  \caption{All parameters are the same as in Figure \ref{figure9} except that $v_0(x)=1.2e^{-3x^2}+ae^{-3(x-\pi)^2}$, with $a=1.2,1.0,0.95$ and $0.8$ from the left to right.  We observe that the a/symmetry of initial data prevails, and the slightly tilted $(u_0,v_0)$ converge to $(u_2^\#,v_2^\#)$ with $L_0=\frac{\pi}{L}$.  However, if We continue asymmetrizing the initial data as in the last column, which is more tilted to the left endpoint than $v_0$, then cells, attracted to the left endpoint, eventually form a single boundary spike on the left.}\label{doublespike}
\end{figure}
On the other hand, as expected from the previous discussion, if we increase the value of the chemosensitivity to $\chi=6$ such that $\chi>\chi_2$, the same initial data stay symmetric and now converge not to the constant steady state $(\bar u,\bar v)$, but to the non-monotone stationary double--boundary spike asymptotically as seen in the first column of Figure \ref{doublespike}.  This state is the one with the lowest energy among the symmetric stationary states and is therefore expected to be chosen for a large class of symmetric initial data. We did not pursue further clarification of the stability of the constant steady state, however, our numerical studies indicate that once $\chi>\chi_2$, the constant steady state $(\bar u,\bar v)$ also becomes unstable for symmetric perturbations while symmetrical multi-bump solutions might be locally stable.

In the next three columns of Figure \ref{doublespike}, we proceed to asymmetrize the initial data for the chemoattractant, by choosing  initial data $v_0(x)=1.2e^{-3x^2}+1.0e^{-3(x-\pi)^2}$, which is slightly tilted to the left end point, and examine how the asymmetry of initial data will affect the selection of asymptotic steady states.

We can summarize this subsection by pointing out that the gradient flow obviously chooses to slide down the steepest descent of the free energy landscape. However, due to the rich number of steady states, it is quite difficult to give precise conditions on the initial data choosing a particular asymptotic state in the whole generality. For instance, giving a precise characterization of the basin of attraction of the compactly supported single bump (two half-bumps) solution for any value of $L$ and $\chi$ is an interesting open problem, in particular, in view of the connection to the Cauchy problem and asymptotic stability of the single bump solutions as discussed in \cite{CHVY}.

\subsection{Metastability and transient behavior}

In the last set of experiments, we want to showcase that slowly variant transient behavior will be present due to the large set of stationary states and the possibly large number of connected components in their compact supports. There are transient states keeping a very similar shape for a very long time giving the impression of false stationarity. This has already been reported in similar aggregation-diffusion problems \cite{BFH} and also discussed in the recent survey \cite{CCY}. We refer to this kind of slow dynamics in the free energy landscape as metastability.

The first simulation in Figure \ref{meta1}-Case (i) shows the asymptotic formation of a single boundary spike given by (\ref{11}) with $\chi=20$ and $L=5$.  The initial data consist of unit total population centered at $x=3$ attracted by chemoattractant concentrated at $x=2$: $u_0(x)=\max\{0,\frac{3}{4}\big(1-(x-3)^2\big)\}$ and $v_0(x)=0.5e^{-2(x-2)^2}.$  We observe that cells are attracted by the chemical and migrate to the left right away and then form a single aggregation centered at $x\approx 2.2$ at $t=1$. This interior spike then endures a meta-stable process for $t\in(1,20.5)$ shifting to the left end very slowly and eventually touching the left endpoint and forming a stationary boundary spike there.  According to our calculations, the steady state is $u(x)=\mathcal A(\cos \omega x-\cos \omega l^*)$ over its support $(0,l^*)$ with $\mathcal A\approx 3.1668$ and $l^*\approx 0.4121$. The numerical simulations fully agree with this formula. The second row plots the decaying of free energy $\mathcal E(u,v)$ given by (\ref{41}). It captures the meta-stable evolution of the interior spike and formation of the single boundary spike and shows that the free energy $\mathcal E(u,v)$ converges to that of the boundary spike, which is the least among all stationary solutions. We also emphasize the staircase behavior of the free energy which is the typical metastability behavior in which the shape of the solution changes dramatically only at points of a high gradient of the decay of the free energy or high values of the free energy dissipation.

\begin{figure}[ht]
        \centering
                \includegraphics[width=0.95\textwidth,height=3.5in]{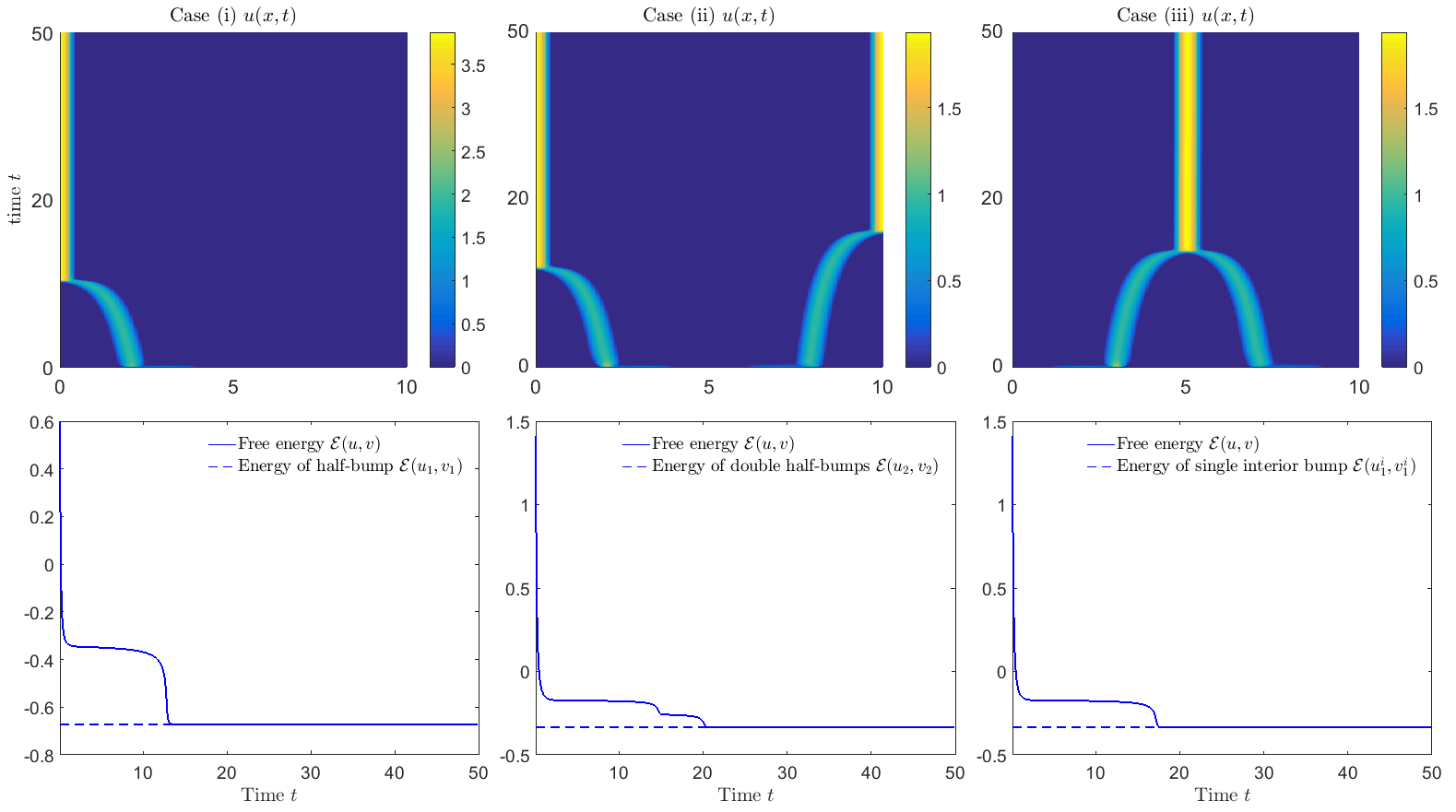}
 \caption{Metastability behavior of \eqref{11} for $\chi=20$ and $L=10$. Case (i): formation of half-bump (single boundary spike). Case (ii): formation of asymmetric two half-bumps. Case (iii):  formation of a single interior bump solution.}
\label{meta1}
\end{figure}

The metastability phenomena are ubiquitous in these models and we show two more cases in Figure \ref{meta1}-Case (ii) and \ref{meta1}-Case (iii). In Figure \ref{meta1}-Case (ii) we show the evolution corresponding to two initial bumps asymmetrically attracted to each other. The initial data are given by $u_0(x)=\max\{0,\frac{3}{8}\big(1-(x-2)^2\big)\}+\max\{0,\frac{3}{8}\big(1-(x-8)^2\big)\}$ and $v_0(x)=1.2e^{-2(x-3)^2}+0.6e^{-2(x-7)^2}$. We observe that the two bumps start attracting and move slowly towards each other. They finally merge into one single stationary single compactly supported bump (two half-bumps) in the middle.

In Figure \ref{meta1}-Case (iii), we take a very similar initial data for the cell density slightly closer initial bumps, however, we take the initial concentration of chemoattractant towards the endpoints of the interval. The precise initial data are given by $u_0(x)=\max\{0,\frac{3}{8}\big(1-(x-3)^2\big)\}+\max\{0,\frac{3}{8}\big(1-(x-7)^2\big)\}$ and $v_0(x)=1.2e^{-2(x-2)^2}+0.6e^{-2(x-8)^2}$. This produces the bumps to separate slowly and get closer and closer to the endpoints of the intervals. The one to the left arrives to zero earlier due to the asymmetry of the initial concentration of chemoattractant which is higher to the left. The one to the right finally also achieves the endpoints leading to convergence to the double symmetric boundary spike solution. The concentration of cell density is symmetric finally due to the symmetry of the cell mass distribution between the two initial bumps. The chemoattractant also achieves a stationary symmetric distribution.

In both cases, the staircase effect in the decay of the free energy is clear. The free energy dissipation is very small except for instances of time in which the shape of the solution changes abruptly from two bumps onto a single one in Case (ii) of Figure \ref{meta1} or the times in which each of the two bumps arrive at the corresponding endpoints of the interval in the Case (iii) of Figure \ref{meta1}.

\section*{Acknowledgements}
JAC was partially supported by EPSRC grant number EP/P031587/1. JAC acknowledges support through the Changjiang Visiting Professorship Scheme of the Chinese Ministry of Education. QW is supported by NSF-China (No. 11501460) and the Fundamental Research Funds for the Central Universities (No. JBK1805001). Z.-A. Wang is supported by the Hong Kong RGC GRF grant PolyU 153031/17P.

\setlength{\bibsep}{0.5ex}
\bibliographystyle{abbrv}
\bibliography{references}

\end{document}